\documentclass{amsart}

\usepackage{amscd,amssymb,amsmath,graphicx,verbatim,xypic}
\usepackage[applemac]{inputenc}
\usepackage[dvips]{hyperref}
\usepackage[TS1,OT1,T1]{fontenc}

\newtheorem{theorem}{Theorem}[section]
\newtheorem{lemma}[theorem]{Lemma}
\newtheorem{corollary}[theorem]{Corollary}
\newtheorem{proposition}[theorem]{Proposition}

\theoremstyle{definition}
\newtheorem{definition}[theorem]{Definition}

\newcommand{\Dcal}{\ensuremath{\mathcal{D}}}
\newcommand{\Xcal}{\ensuremath{\mathcal{X}}}

\newcommand{\Tcal}{\ensuremath{\mathcal{T}}}
\newcommand{\Fcal}{\ensuremath{\mathcal{F}}}
\newcommand{\Acal}{\ensuremath{\mathcal{A}}}
\newcommand{\Bcal}{\ensuremath{\mathcal{B}}}

\newcommand{\Ccal}{\ensuremath{\mathcal{C}}}

\newcommand{\Lbb}{\ensuremath{\mathbb{L}}}
\newcommand{\Zbb}{\ensuremath{\mathbb{Z}}}

\theoremstyle{remark}
\newtheorem{remark}[theorem]{Remark}

\numberwithin{equation}{section}

\begin{document}
\begin{title}{Perverse coherent t-structures through torsion theories}
\end{title}
\begin{author}{Jorge Vit\'oria}\end{author}\thanks{Most of this work was developed at the University of Warwick and supported by FCT (Portugal), research grant SFRH/BD/28268/2006. Later, this project was also supported by DFG (SPP 1388) in Stuttgart and by SFB 701 in Bielefeld. The author would like to thank Steffen Koenig, Qunhua Liu, Dmitriy Rumynin, Jan {\v S}{\v t}ov{\'{\i}}{\v{c}}ek and the anonymous referee for valuable comments on the previous versions of this paper.}

\address{Fakult\"at f\"ur Mathematik, Universit\"at Bielefeld, D-33501 Bielefeld, Germany \newline\indent email:  jvitoria@math.uni-bielefeld.de}
\begin{abstract}
Bezrukavnikov, later together with Arinkin, recovered Deligne's work defining perverse t-structures in the derived category of coherent sheaves on a projective scheme. We prove that these t-structures can be obtained through tilting with respect to torsion theories, as in the work of Happel, Reiten and Smal\o. This approach allows us to define, in the quasi-coherent setting, similar perverse t-structures for certain noncommutative projective planes.
\end{abstract}

\maketitle
\begin{section}{Introduction} 
A t-structure in a triangulated category $\Dcal$ (\cite{BBD}) is a pair of strict full subcategories, $(\Dcal^{\leq 0}, \Dcal^{\geq 0})$, such that, for $\Dcal^{\leq n} := \Dcal^{\leq 0}[-n]$ and $\Dcal^{\geq n} := \Dcal^{\geq 0}[-n]$ ($n\in \mathbb{Z}$),
\begin{enumerate}
\item ${\rm Hom}(X,Y) = 0, \forall X \in \Dcal^{\leq 0},\forall Y \in \Dcal^{\geq 1}$;
\item $\Dcal^{\leq 0} \subseteq \Dcal^{\leq 1}$;
\item For all $X \in \mathcal{D}$, there are $A\in \Dcal^{\leq 0}$, $B\in \Dcal^{\geq 1}$ and a triangle
\begin{equation}\nonumber
A\longrightarrow X \longrightarrow B \longrightarrow A[1].
\end{equation}
\end{enumerate}
The intersection $\Dcal^{\leq 0}\cap \Dcal^{\geq 0}$ is an abelian category (\cite{BBD}), called the heart. Also, it is well known (\cite{KV}) that $\Dcal^{\leq 0}$, called the aisle, determines the t-structure by setting $\Dcal^{\geq 0}=(\Dcal^{\leq 0})^{\perp}[1]$. A t-structure $(\Dcal^{\leq 0},\Dcal^{\geq 0})$ has associated truncation functors $\tau^{\leq i}:\Dcal\rightarrow \Dcal^{\leq i}$, $\tau^{\geq i}:\Dcal\rightarrow \Dcal^{\leq i}$ and cohomological functors $H^i: \Dcal\rightarrow \Dcal^{\leq 0}\cap\Dcal^{\geq 0}$, for all $i\in\Zbb$ (see \cite{BBD} for details). If $\Acal$ is an abelian category, its derived category $\Dcal(\Acal)$ has a standard t-structure, denoted throughout by $(\Dcal_0^{\leq 0},\Dcal^{\geq 0}_0)$, defined by
\begin{equation}\nonumber
\Dcal^{\leq 0}_0:=\{X\in\Dcal: H_0^i(X)=0, \forall i>0\},
\end{equation}
\begin{equation}\nonumber
\Dcal^{\geq 0}_0:=\{X\in\Dcal: H_0^i(X)=0, \forall i<0\},
\end{equation}
where $H_0^i$ is the usual complex cohomology functor. We denote the associated truncation functors by $\tau_0^{\leq i}$ and $\tau_0^{\geq i}$ and the associated cohomological functor is precisely the complex cohomology functor $H_0^i$, for all $i\in\Zbb$. The standard t-structure restricts to the bounded derived category $\Dcal^b(\Acal)$ and we shall use the same notations for the restriction, when appropriate.

Let $\mathbb{K}$ be an algebraically closed field. For a scheme $X$ over $\mathbb{K}$, Arinkin and Bezrukavnikov (\cite{ArBe},\cite{Be}) constructed perverse coherent t-structures in $\Dcal^b(coh(X))$ as follows.  Let $X^{top}$ denote the set of generic points of all closed irreducible subschemes of $X$. A perversity is a map $p:X^{top} \longrightarrow \mathbb{Z}$ satisfying 
\begin{equation}\label{def perv}
y \in \bar{x} \Rightarrow p(y) \geq p(x)\geq p(y) - (dim(x) - dim(y)).
\end{equation}
Note that the image of $p$ has at most $dim(X)+1$ elements. The perverse coherent t-structure associated with $p$ (\cite{ArBe}, \cite{Be}) is defined by the aisle 
\begin{equation}\nonumber
\Dcal^{p,\leq 0}=\left\{F^{\bullet}\in \Dcal^{b}(coh(X)): \forall x \in X^{top},\  Li_{x}^{*}(F^{\bullet}) \in \Dcal_0^{\leq p(x)}(O_{\{x\}}\mbox{-}mod)\right\},
\end{equation}
where  $Li_{x}^{*}$ is the left derived functor of the pullback by  the inclusion of schemes $i_{x}:\left\{x\right\}\longrightarrow X$. In our notation, we identify modules over the residue field $k(x)$ with quasi-coherent sheaves over $\{x\}$. Still, we choose to use the notation $O_{\{x\}}\mbox{-}mod$ for coherent sheaves over $\{x\}$ to be consistent with the notation in \cite{Be}. 
Our main theorem gives an alternative description of this aisle.
\\

\noindent\textbf{Theorem} (Theorem \ref{main perverse})
\textit{Let $X$ be a smooth projective scheme over $\mathbb{K}$, $R=\Gamma_*(X)$ its homogeneous coordinate ring and $p$ a perversity on $X$. Suppose that $R$ is a commutative connected, noetherian, positively graded $\mathbb{K}$-algebra generated in degree 1. Let $\Tcal_i$ denote the torsion class cogenerated  in $Tails(R)$ by $\pi E_i$, where $$E_i=\prod\limits_{\left\{x\in X^{top}:p(x)\leq i\right\}}E^g(R/I_x),$$ with $I_x$ standing for the defining ideal of $x\in X^{top}$ in $R$. Then we have:
\begin{equation}\nonumber
\Dcal^{p,\leq 0} = \left\{ F^{\bullet}\in \Dcal^b(Tails(R)): H_0^i(F^{\bullet})\in \Tcal_j, \forall i>j\right\}\cap \Dcal^b(tails(R)).\\
\end{equation}}

We clarify some notation. Denoting by $O_X$ the structure sheaf of $X$ and by $\Gamma$ the functor of global sections, the homogeneous coordinate ring $R$ is defined by
\begin{equation}\nonumber
R=\Gamma_*(X):=\bigoplus\limits_{n\in\mathbb{Z}} \Gamma(X,O_X(n)).
\end{equation} 
Throughout, $R$ will be assumed to be noetherian. We denote the injective envelope of a graded module $M$ in the category of graded (right) $R$-modules, $Gr(R)$, by $E^g(M)$. The category $Tails(R)$ is the quotient $Gr(R)/Tors(R)$ (we denote the projection functor to this quotient by $\pi$) where $Tors(R)$ is the full subcategory of modules $M$ in $Gr(R)$ such that for all $x$ in $M$ there is $N\geq 0$ with $xR_j=0$, for all $j>N$. 
This category is equivalent to $Qcoh(X)$, the category of quasi-coherent sheaves over $X$, as shown by Serre in \cite{FAC}.  When written in the lower case,  $tails(R) = gr(R)/tors(R)$ denotes the subcategory of finitely generated objects in $Tails(R)$, thus being equivalent to $coh(X)$.  Throughout we will use these equivalences without mention. We will show in section 3 that, for some rings $R$,
\begin{equation}\nonumber
\left\{ X^{\bullet}\in \Dcal^b(Tails(R)): H_0^i(X^{\bullet})\in \Tcal_j, \forall i>j\right\}
\end{equation}
is an aisle of $\Dcal^b(Tails(R))$, obtained through a suitable iteration of tilting with respect to torsion theories (as in the work of Happel, Reiten and Smal\o, \cite{HRS}), for some torsion classes $\{\Tcal_a,...,\Tcal_{a+s}\}$ (see also \cite{AJS}, \cite{Ka} and \cite{Sta} for similar constructions). 

The categories $Tails(R)$ are also defined as above for noncommutative rings $R$, providing a framework to noncommutative projective geometry (\cite{AZ}). Our theorem motivates similar constructions of t-structures in this new setting (see section 5). Given a (noncommutative) graded $\mathbb{K}$-algebra $R$, the noncommutative projective scheme associated with $R$ can be thought of as an abstract space $Proj(R)$ whose category of \textit{quasi-coherent sheaves} (respectively, \textit{coherent sheaves}) is the category $Tails(R)$ (respectively, $tails(R)$) with structure sheaf $\pi R$. Analogously to the commutative case, $\pi$ admits a right adjoint $\Gamma_*$. Artin and Schelter defined in \cite{AS} the class of algebras whose categories of tails play the role of coherent sheaves over noncommutative projective planes. These, called Artin-Schelter (AS for short) regular algebras of dimension $3$, are algebras of global dimension 3, finite Gelfand-Kirillov dimension (in fact equal to 3) and Gorenstein. These algebras have been classified (\cite{AS},\cite{ATV}) by associating to each one a triple $(E,\sigma,L)$, where $E$ is a scheme, $\sigma$ is an automorphism of $E$ and $L$ is an invertible sheaf on $E$. In section 5 we focus on AS-regular algebras of dimension 3 with 3 generators such that  $E$ is a divisor of degree 3 in $\mathbb{P}^2$, $\sigma$ is an automorphism of finite order and $L$ is the restriction of $O_{\mathbb{P}^2}(1)$. These are noetherian algebras and finite over their centres (\cite{ATV2}) - hence fully bounded noetherian. In this setting, we provide an example of a new construction of perverse quasi-coherent t-structures. 

The paper is outlined as follows. Section 2 presents some basics on torsion theories for categories of graded modules. In section 3 we show how to obtain a t-structure by adequately iterating a result of \cite{HRS}, using certain torsion theories of an AB4 abelian category. Section 4 shows how torsion theories come into play when describing perverse coherent t-structures and section 5 applies section 3 to define perverse quasi-coherent t-structures on some noncommutative projective planes.
\end{section}

\begin{section}{Torsion theories for graded modules}
Let $R$ be a noetherian graded ring (not necessarily commutative). The category $Gr(R)$ is a Grothendieck category, admitting injective envelopes which, for a graded module $M$, we denote by $E^g(M)$. The homomorphisms between modules $M$ and $N$ in $Gr(R)$ ($R$-linear,  grading preserving) are denoted by Hom$_{Gr(R)}(M,N)$. The subset of homogeneous elements of $M$ is denoted by $h(M)$. It is clear that $M$ is generated by $h(M)$. Also, for a prime ideal $P$, define $C^g(P)=C(P)\cap h(R)$, where $C(P)$ is the set of regular elements modulo $P$, i.e., the set of elements $x$ of $R$ such that $x+P$ is neither left nor right zero divisor in $R/P$. If $R$ is commutative, then $C(P)=R\setminus P$. The following remark proves to be useful.

\begin{remark}\label{rem}
Given a connected positively graded ring $R$ generated in degree one and a homogeneous prime ideal $P\neq R_+:=\bigoplus_{i\geq 1} R_i$, we have $P_n \neq R_n$ for all $n>1$. Recall that an ideal $P$ of a ring $R$ is prime if and only if for all $x,y\in R$, whenever $xRy\subset P$, either $x$ or $y$ must belong to $P$. Suppose that there is $n_0>1$ such that $P_{n_0}=R_{n_0}$. Note that, since the ring is generated in degree one, we have $P_n=R_n$ for all $n>n_0$. Let $x_1$ be an element in $R_1\setminus P_1$. Since $P$ is prime, there is $r_1\in R$ such that $x_2=x_1r_1x_1 \notin P$. Now, deg$(x_2)\geq 2$ since $R$ is positively graded. Thus we can inductively construct a sequence of elements $(x_n)_{n\in\mathbb{N}}$ none of them lying in $P$ and such that deg$(x_n)>$deg$(x_{n-1})$, yielding a contradiction with the assumption that $P_n=R_n$ for all $n>n_0$.
\end{remark}

We recall the definition of torsion theory.
\begin{definition}
Let $\mathcal{A}$ be an abelian category. A pair of full subcategories $(\mathcal{T}$,$\mathcal{F})$ is said to be a torsion theory if:
\begin{enumerate}
\item Hom$(T,F)=0$, for all $T\in\mathcal{T}$ and $F\in\mathcal{F}$;
\item For all $M\in\mathcal{A}$ there is an exact sequence 
\begin{equation}\nonumber
0\longrightarrow \tau(M) \longrightarrow M\longrightarrow M/\tau{M}\longrightarrow 0,
\end{equation}
where $\tau(M)\in\mathcal{T}$ and $M/\tau(M)\in\mathcal{F}$.
\end{enumerate}
We say that $(\Tcal,\Fcal)$ is a hereditary torsion theory if $\Tcal$ is closed under subobjects. 
\end{definition}

We are particularly interested in (hereditary) torsion theories defined as follows.

\begin{definition}
A torsion theory (or its torsion class) in $Gr(R)$ is said to be cogenerated by an injective object $E$ if the torsion objects are precisely those $M$ satisfying ${\rm Hom}_{Gr(R)}(M,E)=0$. 
\end{definition}

Since $R$ is noetherian, $gr(R)$ is closed under taking subobjects. Thus, torsion theories in $Gr(R)$ restrict to torsion theories in $gr(R)$. For this it is enough to observe that, given a module $M$ in $gr(R)$ and $\tau$ the torsion radical functor induced by a torsion theory in $Gr(R)$, both $\tau(M)$ and $M/\tau(M)$ also lie in $gr(R)$. 

The following lemma proves a useful criterion for graded modules to be torsion with respect to the torsion theory cogenerated by an injective object. The arguments of the proof mimic the ungraded case (see \cite{LM}, lemma 2.5).  

\begin{lemma}\label{lemma homg}
Given graded modules $T$ and $F$ over a graded ring $R$, the following conditions are equivalent:
\begin{enumerate}
\item ${\rm Hom}_{Gr(R)}(T,E^g(F))=0$;
\item $\forall t\in h(T)$, $\forall f\in h(F)\setminus 0$, ${\rm deg}(f)={\rm deg}(t)$, $\exists r\in h(R)$: $tr=0\wedge fr\neq 0$.
\end{enumerate}
\end{lemma}
\begin{proof}
Suppose ${\rm Hom}_{Gr(R)} (T,E^g(F)) \neq 0$. Let $\alpha$ be one of its nonzero elements. Choose $u\in h(T)$ such that $\alpha (u)\neq 0$. Now, $F$ is a graded essential submodule of $E^g(F)$, i.e., given any nontrivial graded submodule of $E^g(F)$, its intersection with $F$ is nontrivial. Hence there is $s\in h(R)$ such that $0\neq \alpha(u)s= \alpha(us) \in F$. If we choose $t=us$ and $f=\alpha(us)$, they are homogeneous of the same degree and clearly, given $r\in R$, if $tr=0$ then $fr=0$.

Suppose now that (2) is false, i.e., there are $t\in T$ and $f\in F\setminus\left\{0\right\}$ homogeneous of the same degree such that for all $r\in h(R)$, if $tr=0$ then $fr=0$. Then, there is a well defined nonzero graded homomorphism
\begin{equation}\nonumber
tR\longrightarrow F,\ tr\mapsto fr
\end{equation}
since $\left\langle h(R)\right\rangle = R$. Since $E^g(F)$ is an injective object in the category of graded modules, we can find a nonzero graded homomorphism from $T$ to $E^g(F)$.
\end{proof}

The following corollary shows how to reformulate a statement about graded localisation in terms of torsion.

\begin{corollary}\label{corollary deg zero}
Let $R$ be a commutative graded ring, $P$ a homogeneous prime ideal in $R$ and $S=h(R\setminus P)$. Given $M$ a graded $R$-module then $(S^{-1}M)_0 = 0$ if and only if ${\rm Hom}_{Gr(R)}(M,E^g(R/P))=0$.
\end{corollary}
\begin{proof}
This follows from the fact $(S^{-1}M)_0 = 0$ is equivalent, by definition of graded localisation, to condition (2) of the above lemma with $T=M$ and $F=R/P$.
\end{proof}

We will now look at rigid torsion theories (\cite{NVO}). We shall consider the following subset of Hom$_R(M,N)$: $\overline{{\rm Hom}}(M,N):=\bigoplus\limits_{i\in\mathbb{Z}}{\rm Hom}_{Gr(R)}(M,N(i))$.  

\begin{definition}
We say that a torsion theory in $Gr(R)$ is \textit{rigid} if the class of torsion modules (equivalently, the class of torsion-free modules) is closed under shifts of the grading. The \textit{rigid torsion theory cogenerated by an injective object $E$ in $Gr(R)$} is defined such that a module $M$ is torsion if $\overline{{\rm Hom}}(M,E)=0$.
\end{definition}

\begin{remark}\label{lemma big hom}
Observe that, given graded modules $T$ and $F$ over a graded ring $R$,  $\overline{{\rm Hom}}(T,E^g(F))=0$ if and only if Hom$_{Gr(R)}(T(j),E^g(F))=0$ for all $j\in\mathbb{Z}$. This allows us to get analogues of lemma \ref{lemma homg} and corollary \ref{corollary deg zero} as follows: 
\begin{enumerate}

\item The following conditions are equivalent: \begin{itemize} \item$\overline{{\rm Hom}}(T,E^g(F))=0$ \item $\forall t\in h(T)$, $\forall f\in h(F)\setminus 0$, $\exists r\in h(R)$ such that $tr=0$ and $fr\neq 0$.\end{itemize}

\item If $R$ is commutative, $P$ a homogeneous prime ideal in $R$, $S=h(R\setminus P)$ and $M$ a graded $R$-module, then we have that $S^{-1}M = 0$ if and only if $\overline{{\rm Hom}}(M,E^g(R/P))=0$.

\end{enumerate}
\end{remark}

It is, in fact, possible to get a more general statement, including some noncommutative rings, by comparing this rigid torsion theory with the torsion theory associated to a multiplicative set. For a homogeneous right ideal $J$ of a graded ring $R$ we use the notation $J\triangleleft_{rg} R$ and, given $r\in R$, we define a right ideal
\begin{equation}\nonumber
r^{-1}J:=\left\{a\in R: ra\in J\right\}.
\end{equation}
Recall the following result (\cite{NVO}, Proposition A.II.9.11). 

\begin{proposition}\label{proposition gabriel filter}
Let $R$ be a graded ring, $S$ a multiplicative subset contained in $h(R)$. Then the class of modules $M$ such that there is $J\in\mathbb{L}_S$ with $MJ=0$, where
\begin{equation}\nonumber
\mathbb{L}_S = \left\{J\triangleleft_{rg} R: r^{-1}J\cap S\neq\varnothing,\ \forall r\in h(R)\right\},
\end{equation}
is a torsion class for a rigid torsion theory in $Gr(R)$.
\end{proposition}

$\mathbb{L}_S$ as above is said to be a graded Gabriel filter for that torsion theory. If $S=C^g(P)$ for some homogeneous prime ideal $P$, then we denote the filter by $\mathbb{L}_P$. The rigid torsion theory associated to an injective graded module $E$ also has an associated graded Gabriel filter given by:
\begin{equation}\nonumber
\mathbb{L}_E^r=\left\{J\triangleleft_{rg} R: \overline{{\rm Hom}}(R/J,E)=0\right\}.
\end{equation}

In fact, hereditary rigid torsion theories are in bijection with graded Gabriel filters (\cite{NVO}, Lemma A.II.9.4). Thus the graded Gabriel filter determines the torsion theory and vice-versa. The following two supporting lemmas will be useful in proving the main theorem of this section.

\begin{lemma}\label{there is a homog}
Let $R=\oplus_{i\geq 0} R_i$ be a noetherian graded ring, $P$ a homogeneous prime ideal and $J$ a homogeneous right ideal of $R$. If $J\cap C(P)\neq \varnothing$ then $J\cap C^g(P)\neq \varnothing$.
\end{lemma}
\begin{proof}
Recall that, in a prime noetherian ring, an element is left regular if and only if it is right regular (see \cite{Her} for details).  
Therefore we can regard $C(P)$ as the set of elements $x\in R$ such that $x+P$ is right regular in $R/P$. Let $c\in J\cap C(P)$ and consider its homogeneous decomposition in $J$: $c=c_{i_1}+c_{i_2}+...+c_{i_n}$ where $c_{i_j}\in J\setminus\left\{0\right\}\cap R_{i_j}$. If $c_{i_1}\in C(P)$, we are done. If not, by definition, there is $r_1\in R\setminus P$ such that $c_{i_1}r_1\in P$. Moreover, the choice of $r_1$ can be made in $h(R\setminus P)$, since $P$ is a homogeneous ideal. %decompose cr_1...
Clearly $cr_1+P$ is right regular in $R/P$ and, thus, $c^{(1)}:=cr_1-c_{i_1}r_1\notin P$. We iterate this argument by looking at the first homogeneous component of $c^{(1)}$ (which is $c_{i_2}r_1$). Assume now that this $n$-step iteration does not yield a homogeneous element in $J\cap C(P)$. Then, this argument gives a sequence $r_1,..., r_n$ of elements in $R$ such that $cr_1...r_n\in P$, which is a contradiction to $c$ being regular modulo $P$.
\end{proof}

\begin{lemma}\label{filter char}
Let $R$ be a positively graded noetherian ring, $J$ a right ideal of $R$ and $M$ a graded right $R$-module. Then $J$ lies in $\Lbb_{E^g(M)}^r$ if and only if $m(x^{-1}J)\neq 0$ for all $m$ in $h(M)\setminus\left\{0\right\}$ and $x$ in $h(R)$.
\end{lemma}
\begin{proof}
Note that $J\in \Lbb_{E^g(M)}^r$ if and only if, for every cyclic graded submodule $C$ of $R/J$, $\overline{{\rm Hom}}(C,E^g(M))=0$. Now, it is easy to see that the graded cyclic submodule generated by $x+J$, for some $x\in h(R)$, is isomorphic to $R/x^{-1}J$. Now, of course, $\overline{{\rm Hom}}(R/x^{-1}J,M)=0$ if and only if, for all $m\in h(M)\setminus\left\{0\right\}$, $m(x^{-1}J)\neq 0$.
\end{proof}

The following theorem is a graded version of the main result in \cite{LM}.

\begin{theorem}\label{theorem two torsions}
Let $P$ be a homogeneous prime ideal of a graded ring $R$ and $R/P$ noetherian. Let $M$ be a graded right $R$-module. Then $M$ is torsion with respect to the rigid torsion theory associated with $C^g(P)$ if and only if $M$ is torsion with respect to the rigid torsion theory associated to $E^g(R/P)$.
\end{theorem}
\begin{proof}
We will prove that the graded Gabriel filters of both torsion theories coincide. Let $E=E^g(R/P)$ and $J\in\mathbb{L}_{E}^r$. By lemma \ref{filter char} this is equivalent to say that for all $0\neq a+P\in h(R/P)$ and for all $x\in h(R)$, $(a+P)(x^{-1}J)\neq 0$. This means that, for any choice of $a+P\in h(R/P)$, the two-sided ideal 
\begin{equation}\nonumber
K:=(R/P)(a+P)((x^{-1}J+P)/P)\triangleleft R/P
\end{equation}
is nonzero. Since $R/P$ is a two-sided noetherian ring, it is well-known that every two-sided ideal is essential as a right ideal and, therefore, by Goldie's theorem for graded rings 
 (see \cite{GS}, theorem 4), we have that $K$ contains a homogenous regular element $c+P$. Now $c+P=(ba+P)(t+P)$ where $t\in x^{-1}J$. Clearly, $t+P$ is right regular and, since $R/P$ is a prime noetherian ring, it is regular. Therefore, we conclude that $t\in x^{-1}J\cap C(P)$. By lemma \ref{there is a homog}, we also have that $x^{-1}J\cap C^g(P)\neq \varnothing$ and thus $J\in \Lbb_P$.

Conversely, suppose $J\in \mathbb{L}_{P}$ and let $a,b\in h(R), \ b\notin P$. By hypothesis, $a^{-1}J\cap C^g(P)\neq \varnothing$. Let $z$ be one of its elements. Then, clearly, $az\in J$ and $bz\notin P$. Again, by remark \ref{lemma big hom}, the result follows.
\end{proof}

In the commutative positively graded case, however, the torsion theory associated with the injective module $E^g(R/P)$ coincides with the one associated to the multiplicative set $h(R\setminus P)$ by the following well-known result, the proof of which we, thus, omit. 

\begin{proposition}
Let $R$ be a commutative noetherian positively graded connected $\mathbb{K}$-algebra generated in degree 1, $P$ a homogeneous prime ideal in $R$ not equal to the irrelevant ideal and $M$ a graded $R$-module. Then, for $S=h(R\setminus P)$, $S^{-1}M=0$ if and only if $(S^{-1}M)_0 =0$. 
\end{proposition}

This shows that under the conditions of the proposition above, the torsion theory cogenerated by $E^g(R/P)$ is automatically rigid. This statement, however, can be proved without assuming commutativity.

\begin{lemma}\label{cogen is rigid}
Let $R$ be a noetherian positively graded connected $\mathbb{K}$-algebra generated in degree 1, $P$ a homogeneous prime ideal in $R$ not equal to the irrelevant ideal and $M$ a right graded $R$-module. Then, ${\rm Hom}_{Gr(R)}(M,E^g(R/P))=0$ if and only if $\overline{{\rm Hom}}(M,E^g(R/P))=0$.
\end{lemma}
\begin{proof}
One direction is clear. Suppose $\overline{{\rm Hom}}(M,E^g(R/P))\neq 0$. Then by remark \ref{lemma big hom} there is $m\in h(M)$ such that $Ann(m)\cap C^g_l(P) =\varnothing$, where $Ann(m)$ stands for right annihilator of $m$ and $C^g_l(P)$ stands for homogeneous left regular elements mod $P$. We want to prove ${\rm Hom}_{Gr(R)}(M,E^g(R/P))\neq 0$. By remark \ref{rem}, there is a homogeneous element in $R\setminus P$ in each positive degree and thus, by lemma \ref{lemma homg}, ${\rm Hom}_{Gr(R)}(M,E^g(R/P))\neq 0$ is equivalent to the existence of $\tilde{m}\in h(M_{\geq 0})$ such that $Ann(\tilde{m})\cap C^g_l(P) =\varnothing$.  

Note that the irrelevant ideal, $R_+$, is a homogeneous maximal ideal containing $P$. Clearly, $R_+/P$ is an essential ideal in the graded prime Goldie ring $R/P$ and so, by graded Goldie's theorem (see \cite{GS}, theorem 4), it containing a regular element. This means that there is a homogeneous regular element of positive degree in $R/P$ and thus $C^g(P)_{\geq k}\neq \varnothing$ for all $k\in\mathbb{N}$. Choose $s\in C^g(P)$ such that deg$(ms)\geq 0$. Note that if there is $a\in Ann(ms)\cap C^g_l(P)$, then $sa\in Ann(m)\cap C^g_l(P)$ yielding a contradiction. Therefore, take $\tilde{m}=ms$ and we are done.
\end{proof}
\begin{remark}
We summarise the results of this section. If $R$ is a noetherian positively graded connected $\mathbb{K}$-algebra generated in degree 1, $P$ a homogeneous prime ideal not equal to the irrelevant ideal $R_+$ and $M$ a graded $R$-module, then the following are equivalent:
\begin{enumerate}
\item {\rm Hom}$_{Gr(R)}(M,E^g(R/P))=0$;
\item $\overline{{\rm Hom}}(M,E^g(R/P))=0$;
\item $M$ is torsion with respect to $C^g(P)$.
\end{enumerate}
If, furthermore, $R$ is commutative and $S=h(R\setminus P)$, then (1), (2) and (3) are equivalent to $S^{-1}M=0$ and to $(S^{-1}M)_0=0$.
\end{remark}
\end{section}

\begin{section}{t-structures via torsion theories}
Recall that an abelian category $\Acal$ is said to be AB4 if it admits arbitrary coproducts and they are exact. It is well known that, under this assumption, $\Dcal(\Acal)$ admits arbitrary coproducts as well. %Derived categories, Hoshino, http://www.u-gakugei.ac.jp/~miyachi/papers/DC1.pdf
In this section we will show that, for $a\in\mathbb{Z}$, $n\in\mathbb{N}$ and certain ordered sets (indexed by a string of integers of length $n$ starting at $a$) of hereditary torsion classes in an AB4 abelian category $\Acal$
\begin{equation}\nonumber
S = \left\{\Tcal_a,\Tcal_{a+1},...,\Tcal_{a+n-1}\right\}\ \ {\rm with}\ \ \Tcal_a\supseteq \Tcal_{a+1}\supseteq \Tcal_{a+2}\supseteq ...\supseteq \Tcal_{a+n-1}=0,
\end{equation}
the following subcategory is the aisle of a t-structure in $\Dcal^b(\Acal)$,
\begin{equation}\nonumber
\Dcal^{S,\leq 0} := \left\{ X^{\bullet}\in \Dcal^b(\mathcal{A}): \ H_0^i(X^{\bullet})\in \Tcal_j,\ \forall i>j\right\}.
\end{equation}

\begin{remark}\label{aisle contained}
Clearly such a category is a subcategory of $\Dcal_0^{\leq a+n-1}$, a shift of the standard aisle. This follows from the assumption that $\Tcal_{a+n-1}=0$.
\end{remark}

Our proof relies on a suitable iteration of a well-known theorem, originally due to Happel, Reiten and Smal\o \ (\cite{HRS}, Proposition 2.1). We present here a slightly modified version of that result, as stated by Bridgeland (\cite{Br}). Recall that a t-structure $(\Dcal^{\leq 0},\Dcal^{\geq 0})$ in a triangulated category $\Dcal$ is said to be bounded if
%$$\bigcap\limits_{n\in\mathbb{Z}}\Dcal^{\leq n}=0=\bigcap\limits_{n\in\mathbb{Z}}\Dcal^{\geq n}$$
$$\bigcup\limits_{n\in\mathbb{Z}}\Dcal^{\leq n}=\Dcal=\bigcup\limits_{n\in\mathbb{Z}}\Dcal^{\geq n}.$$

\begin{theorem}[Happel, Reiten, Smal\o, \cite{HRS}, Bridgeland, \cite{Br}]\label{HRS tilting}
Let $\mathcal{A}$ be the heart of a bounded t-structure in a triangulated category  $\mathcal{D}$. Suppose that  $(\mathcal{T},\mathcal{F})$ is a torsion theory in $\mathcal{A}$ and that $H^i$ denotes the i-th cohomology functor with respect to $\mathcal{A}$. Then $(\mathcal{D}^{\leq 0}, \mathcal{D}^{\geq 0})$ is a t-structure in $\mathcal{D}$, where
\begin{equation}\nonumber
\mathcal{D}^{\leq 0}=\left\{E\in\mathcal{D}: H^i(E)=0,\ \forall i>0, H^0(E)\in \mathcal{T}\right\}
\end{equation}
\begin{equation}\nonumber
\mathcal{D}^{\geq 0}=\left\{E\in\mathcal{D}: H^i(E)=0,\ \forall i<-1, H^{-1}(E)\in \mathcal{F}\right\}.
\end{equation}
Moreover, $(\Fcal[1],\Tcal)$ is a torsion theory in $\Dcal^{\leq 0}\cap\Dcal^{\geq 0}$.
\end{theorem}

The new t-structure (or its heart) obtained in the theorem will be called \textit{the HRS-tilt of $\mathcal{A}$ with respect to $(\mathcal{T},\mathcal{F})$}. We will need a few technical lemmas about this new heart. 
We start with an useful observation about some of its morphisms.

\begin{lemma}\label{epi}
Let $\mathcal{A}$ be the heart of a bounded t-structure in $\mathcal{D}$, a triangulated category. Suppose that  $(\mathcal{T},\mathcal{F})$ is a hereditary torsion theory in $\mathcal{A}$ and that $\Bcal$ is the corresponding heart of the HRS-tilt. For an object $T$ in $\Tcal$ we have:
\begin{enumerate}
\item a morphism $f:T\longrightarrow N$ is an epimorphism in $\Bcal$ if and only if $N$ lies in $\Tcal\subset \Acal$ and $f$ is an epimorphism in $\Acal$;
\item a morphism $f:M\longrightarrow T$ is a monomorphism in $\Bcal$ if and only if $M$ lies in $\Tcal\subset \Acal$ and $f$ is a monomorphism in $\Acal$.
\end{enumerate}
\end{lemma}
\begin{proof}
(1) Let $T$ be an object in $\Tcal$, $N$ in $\Bcal$ and $f$ an epimorphism in Hom$_\Bcal(T,N)$. Let $C\in\Tcal$ and $F[1]\in\Fcal[1]$ be such that we have a short exact sequence in $\Bcal$
\[
0\longrightarrow F[1] \longrightarrow N \longrightarrow C \longrightarrow 0
\]
since $(\Fcal[1],\Tcal)$ is a torsion theory in $\Bcal$.
Consider the following commutative diagram
\begin{equation}\nonumber
\xymatrix{T\ar[r]^f\ar[d]&N\ar[r]\ar[d]&K[1]\ar[r]& T[1]\ar[d]\\ T\ar[r]\ar[d]_f&C\ar[r]\ar[d]&L[1]\ar[r]& T[1]\ar[d]\\ N\ar[r]&C\ar[r]& F[2]\ar[r]& N[1]}
\end{equation}
where the rows are triangles in $\Dcal$ and where $K$ stands for the kernel of $f$ in $\Bcal$ and $L$ for the kernel in $\Bcal$ of the composition of $f$ with the epimorphism $N\rightarrow C$ in $\Bcal$. Note, however, that since $T\in\Tcal$ and $\Tcal$ is a torsion-free class in $\Bcal$ (thus closed under subobjects in $\Bcal$), both $K$ and $L$ lie in $\Tcal$. The octahedral axiom applied to the above diagram yields (after an adequate rotation) the triangle
\[
F\longrightarrow K \longrightarrow L \longrightarrow F[1]
\]
which induces a short exact sequence in $\Acal$, where $F$ is, therefore, a subobject of $K$ in $\Acal$. Since $\Tcal$ is a hereditary torsion class in $\Acal$ and $K$ lies in $\Tcal$, we conclude that $F$ lies in $\Tcal$ and hence it is zero, proving that $N$ is isomorphic to $C$, an object of $\Tcal$.

Conversely, if $N$ lies in $\Tcal$ and $f$ in Hom$_\Acal(T,N)$ is an epimorphism, then its kernel in $\Acal$ also lies in $\Tcal$ ($\Tcal$ is hereditary). Thus, the short exact sequence defined by $f$ is also a short exact sequence in $\Bcal$ and $f$ is an epimorphism in $\Bcal$.

(2) Given a monomorphism  $g\in$ Hom$_\Bcal(M,T)$ for some $M\in\Bcal$ and $T\in\Tcal$, we easily see that $M\in\Tcal$ (since $\Tcal$ is a torsion-free class in $\Bcal$) and, by (1), that the cokernel of $g$ in $\Bcal$ lies in $\Tcal$. Thus, we have a short exact sequence in $\Acal$:
\[
0\longrightarrow M\longrightarrow T\longrightarrow coker(g) \longrightarrow 0.
\]

Conversely, if $M$ lies in $\Tcal$ and $f$ in Hom$_\Acal(T,M)$ is a monomorphism, then its cokernel in $\Acal$ also lies in $\Tcal$ ($\Tcal$ is a torsion class). Thus, the short exact sequence defined by $f$ is also a short exact sequence in $\Bcal$ and $f$ is a monomorphism in $\Bcal$.
\end{proof}

\begin{definition}
Let $\Acal$ be an abelian category. We say that the heart $\Bcal$ of a bounded t-structure in $\Dcal^b(\Acal)$ (or the t-structure itself) is  \textit{uniformly bounded} if there are $m,n\in\mathbb{Z}$ such that $\Bcal\subseteq \Dcal^{\leq m}_0\cap\Dcal^{\geq n}_0$.
A family of objects $(Z_k)_{k\in K}$ in $\Dcal^b(\Acal)$ is \textit{uniformly bounded with respect to a t-structure} $(\Dcal^{\leq 0},\Dcal^{\geq 0})$ if there are $m,n\in\mathbb{Z}$ such that $Z_k\in\Dcal^{\leq m}\cap \Dcal^{\geq n}$ for all $k\in K$. 
\end{definition}

Note that, in an AB4 abelian category $\Acal$, a family $(Z_k)_{k\in K}$ is uniformly bounded with respect to the standard t-structure if and only if its coproduct lies in $\Dcal^b(\Acal)$. This follows from the fact that the standard cohomology commutes with coproducts, since coproducts in $\Acal$ are exact. We now show similar statements for certain hearts.

\begin{lemma}\label{t-exact coprod} Let $\Acal$ be an AB4 abelian category, $(\Dcal^{\leq 0},\Dcal^{\geq 0})$ a uniformly bounded t-structure in $\Dcal^b(\Acal)$ and  $\Bcal$ its heart. Then, $\Bcal$ is cocomplete if and only if existing coproducts are t-exact in $\Dcal^b(\Acal)$.
\end{lemma}
\begin{proof}
Since $\Bcal$ is uniformly bounded and coproducts commute with standard cohomologies, small coproducts of elements in $\Bcal$ exist in $\Dcal^b(\Acal)$. If existing coproducts are t-exact in $\Dcal^b(\Acal)$ then, clearly,  $\Bcal$ is cocomplete. 

Conversely, we first observe that right t-exactness is automatic. This follows from the fact that $\Dcal^{\leq n}$ is left Hom-orthogonal to $\Dcal^{\geq n+1}$ (see, for example, lemma 1.3 in \cite{AJSS}). To prove left t-exactness, let $(Y_k)_{k\in K}$ a family of objects in $\Dcal^{\geq 0}$ such that its coproduct, call it $Y$, lies in $\Dcal^b(\Acal)$.  We shall prove that $Y\in \Dcal^{\geq 0}$. As usual, $\tau^{\leq n}$, $\tau^{\geq n}$ denote the truncation functors and $H^n$ the cohomological functors, for all $n\in\Zbb$, with respect to the fixed t-structure $(\Dcal^{\leq 0},\Dcal^{\geq 0})$. 
Denote by $Y_k^{i}:=\tau^{\geq i}(Y_k)$ and by $B_k^i:=H^i(Y_k)$, for any $k\in K$ and $i\in\Zbb$, the respective truncation and cohomology functors. We have the following sequence of triangles (often called a Postnikov tower or, in certain contexts, a Harder-Narasimham filtration) for each $Y_k$, where $m_k\geq 0$ is the maximal degree for which cohomology does not vanish.  
\begin{equation}\nonumber
\xymatrix{Y_k=Y_k^0 \ar[rr]&& Y_k^{1}\ar[dl]^{[1]}\ar[r]& \dots \ar[r]& Y_k^{m_k}\ar[rr]\ar[ld]^{[1]}&& 0\ar[dl]^{[1]}\\ & B^{0}_k\ar[ul] &&\dots \ar[lu]&& B^{m_k}_k[-m_k]\ar[ul]}.
\end{equation}

Observe that, since $Y\in\Dcal^b(\Acal)$ and $\Bcal$ is uniformly bounded,  
the set $\left\{m_k:k\in K\right\}$ has a maximum - call it $m$.
By extending trivially each of these sequences of triangles to sequences with $m$ triangles and since the coproduct of triangles is a triangle, % see Keller http://atlas.mat.ub.es/grgta/articles/Keller.pdf 
we may consider the coproduct of these sequences componentwise and this yields a (finite) Postnikov tower for $Y$, since $\Bcal$ is cocomplete.  Thus, the Postnikov tower looks as follows:
\begin{equation}\nonumber
\xymatrix{Y\ar[rr]&& \coprod\limits_{k\in K}Y_k^{1}\ar[dl]^{[1]}\ar[r]& \dots \ar[r]& \coprod\limits_{k\in K}Y_k^{m}\ar[rr]\ar[ld]^{[1]}&& 0\ar[dl]^{[1]}\\ & \coprod\limits_{k\in K}B^{0}_k\ar[ul] &&\dots \ar[lu]&& \coprod\limits_{k\in K} B^{m}_k[-m]\ar[ul]}.
\end{equation}
Now, it is clear that $$\coprod\limits_{k\in K}Y_k^{m}\cong \coprod\limits_{k\in K}B^{m}_k[-m] \in \Bcal[-m]\subset \Dcal^{\geq m},$$ since $\Bcal$ is cocomplete. The triangle
\begin{equation}\nonumber
\coprod\limits_{k\in K}B_k^{m-1}[-m+1]\longrightarrow \coprod\limits_{k\in K}Y_k^{m-1}\longrightarrow\coprod\limits_{k\in K}Y_k^{m}\longrightarrow (\coprod\limits_{k\in K}B_k^{m-1}[-m+1])[1]
\end{equation}
shows that $$\coprod\limits_{k\in K}Y_k^{m-1}\in \Dcal^b(\Acal)\ \ \ {\rm and}\ \ \  \coprod\limits_{k\in K}Y_k^{m-1}\in \Dcal^{\geq m-1}.$$ Iterating this argument we get that $Y$ lies in $\Dcal^{\geq 0}$.
\end{proof}

\begin{remark}\label{chain uni bdd}
Note that a uniformly bounded family of objects $(Z_k)_{k\in K}$ in  $\Dcal^b(\Acal)$ with respect to a uniformly bounded t-structure $(\Dcal^{\leq 0},\Dcal^{\geq 0})$ is also uniformly bounded with respect to the standard t-structure in $\Dcal^b(\Acal)$. This can be checked using Postnikov towers, in a similar argument to the one used in the proof above.
\end{remark}

\begin{corollary}\label{coprods commute}
Let $\Acal$ be an AB4 abelian category and $(\Dcal^{\leq 0},\Dcal^{\geq 0})$ a uniformly bounded t-structure in $\Dcal^b(\Acal)$ with cocomplete heart $\Bcal$ and cohomological functors $H^i$, for all $i\in\Zbb$. If $(Z_k)_{k\in K}$ is a uniformly bounded family of objects with respect to $(\Dcal^{\leq 0},\Dcal^{\geq 0})$, then 
\begin{equation}\nonumber
H^i(\coprod\limits_{k\in K} Z_k)=\coprod\limits_{k\in K}H^i(Z_k).
\end{equation}
\end{corollary}
\begin{proof}
Let $Z$ denote the coproduct of the family $(Z_k)_{k\in K}$. By remark \ref{chain uni bdd}, $Z$ lies in $\Dcal^b(\Acal)$. The argument in the proof of lemma \ref{t-exact coprod} shows that a Postnikov tower of $Z$ can be obtained as the coproduct over $K$ of Postnikov towers of each $Z_k$. Since the lower vertices of a Postnikov tower are unique up to isomorphism (they are shifts of the cohomologies with respect to the fixed t-structure), the result follows.
\end{proof}

We will consider torsion classes with a special property. Recall that an object $X$ in $\Dcal(\Acal)$ is compact if the functor Hom$_{\Dcal(\Acal)}(X,-)$ commutes with coproducts.
\begin{definition}
A subcategory $\Tcal$ of a heart $\Bcal$ of $\Dcal^b(\Acal)$ is \textit{compactly generated} in $\Dcal^b(\Acal)$ if every object in $\Tcal$ is the colimit in $\Bcal$ of a family of subobjects in $\Bcal$ which are compact when regarded as objects in $\Dcal^b(\Acal)$.
\end{definition}

Similar results to the following lemma have been obtained by Colpi and Fuller in \cite{CF} for HRS-tilts of the standard heart.

\begin{lemma}\label{cocomplete heart} Let $\Acal$ be an AB4 abelian category. Suppose that $(\Tcal,\Fcal)$ is a hereditary torsion theory in the heart $\Bcal$ of a uniformly bounded t-structure in $\Dcal^b(\Acal)$, with $\Bcal$ cocomplete and $\Tcal$ a subcategory of $\Bcal$ compactly generated in $\Dcal^b(\Acal)$. Then the HRS-tilt of $\Bcal$ with respect to $(\Tcal,\Fcal)$ is uniformly bounded and cocomplete.
\end{lemma}
\begin{proof}
First we show that $\Fcal$ is closed under coproducts. Let $(Y_i)_{i\in I}$ be a family of objects in $\Fcal$ and let $X\in\Tcal$. 
Let $(X_j)_{j\in J}$ be a family of compact subobjects of $X$ in $\Bcal$ (and, hence, in $\Tcal$, by lemma \ref{epi}) such that $X=\varinjlim\limits_{j\in J}X_j$. Now,
\begin{equation}\nonumber
{\rm Hom}_{\Bcal}(\varinjlim_{j\in J} X_j, \coprod\limits_{i\in I} Y_i)=\varprojlim_{j\in J} {\rm Hom}_{\Bcal}(X_j,\coprod\limits_{i\in I} Y_i) =\varprojlim_{j\in J} \coprod\limits_{i\in I}\ {\rm Hom}_{\Bcal}(X_j,Y_i)=0
\end{equation}
since the $X_i$ are compact in $\Dcal^b(\Acal)$ and $\Bcal$ is a full subcategory of $\Dcal^b(\Acal)$. 

Let $\Ccal$ be the HRS-tilt of $\mathcal{B}$ with respect to $(\mathcal{T},\mathcal{F})$, $(Z_k)_{k\in K}$ a family of objects in $\Ccal$ and $Z$ its coproduct. We denote $H^i$ ($i\in\mathbb{Z}$) the cohomological functors defined by the t-structure of which $\Bcal$ is the heart.  
For any $C\in\Ccal$, there is a triangle
\begin{equation}\nonumber
H^{-1}(C)[1]\longrightarrow C\longrightarrow H^0(C) \longrightarrow H^{-1}(C)[2]
\end{equation}
which shows that, since $\Bcal$ is uniformly bounded, then so is $\Ccal$. This shows that $Z$ lies in $\Dcal^b(\Acal)$. Since $\Bcal$ is cocomplete, Corollary \ref{coprods commute} shows that coproducts in $\Dcal^b(\Acal)$ commute with $H^i$, for all $i\in\mathbb{Z}$.  It is then clear that $H^i(Z)=0$ for all $i\neq 0,-1$ and, since both $\Tcal$ and $\Fcal$ are closed under coproducts, $H^{-1}(Z)\in\Fcal$ and $H^0(Z)\in\Tcal$, showing that $Z\in\Ccal$ and, thus, completing the proof.
\end{proof}

The next result can be found in Dickson's work \cite{Dickson} or in Stenstr\"om's book \cite{St}, usually also assuming that the underlying abelian category is complete and well-powered (i.e., the class of subobjects of a given object form a set). It is clear from the proof, however, that the completeness assumption is not necessary and that the well-powered condition on the abelian category can be made weaker. Given a full subcategory $\Xcal$ of an abelian category $\Acal$, we will say that $\Acal$ is \textit{$\Xcal$-well-powered} if for any given object of $\Acal$, the class of its subobjects lying in $\Xcal$ form a set.

\begin{lemma}\label{tor char}
Let $\Acal$ be a cocomplete abelian category and $\Tcal$ a full subcategory of $\Acal$. Assume that every object in $\Tcal$ is the colimit of a directed set of subobjects lying in a subcategory $\Xcal\subseteq \Tcal$ such that $\Acal$ is $\Xcal$-well-powered. Then $\Tcal$ is a torsion class if and only if it is closed under extensions, images and coproducts.
\end{lemma}

The assumptions in the lemma ensure that any object $Y$ in $\Acal$ has a maximal subobject lying in $\Tcal$: it is the colimit of all subobjects of $Y$ that lie in $\Xcal$ (which form a set by hypothesis). 

The following lemma is crucial in the proof of theorem \ref{main it}. Although we need the technical assumptions in the lemma for this abstract setting, they are harmless for the purpose of our applications (see remark \ref{good cats} at the end of this section).

\begin{lemma}\label{new torsion class}
Let $\Acal$ be an AB4 abelian category and consider a uniformly bounded t-structure in $\Dcal^b(\Acal)$ with a cocomplete heart $\Bcal$. Suppose that $(\Tcal,\Fcal)$ is a hereditary torsion theory in $\Bcal$ such that $\Tcal$ is a subcategory of  $\Bcal$ compactly generated in $\Dcal^b(\Acal)$ and such that the compact objects of $\Dcal^b(\Acal)$ lying in $\Tcal$ form a set. Let $\Ccal$ denote the HRS-tilt of $\Bcal$ with respect to $(\Tcal,\Fcal)$. Then, a subcategory $\Tcal_1$ of $\Tcal$ is a hereditary torsion class in $\Ccal$ if and only if it is a hereditary torsion class in $\Bcal$, in which case the subcategory $\Tcal_1$ of $\Ccal$ is, moreover, compactly generated in $\Dcal^b(\Acal)$.
\end{lemma}
\begin{proof}
Suppose first that $\Tcal_1$ is a hereditary torsion class in $\Bcal$. We first show that $\Tcal_1$ is closed under extensions, coproducts and epimorphic images in $\Ccal$. The first two hold trivially (exact sequences in $\mathcal{C}$ are precisely the distinguished triangles of $\Dcal(\mathcal{A})$ that lie in $\mathcal{C}$ and if the two outer terms lie in $\Bcal$ then so does the middle one) since $\Tcal_1$ is a torsion class in $\Bcal$. To see it is closed under epimorphisms, we use lemma \ref{epi}. Indeed, if $f:T\longrightarrow C$ is an epimorphism in $\Ccal$ with $T\in\Tcal_1\subseteq \Tcal$ and $C\in\Ccal$, we have that, by lemma \ref{epi}, $C\in \Tcal$ and $f$ is an epimorphism in $\Bcal$. Since $\Tcal_1$ is a torsion class in $\Bcal$, $C$ must lie in $\Tcal_1$. Finally, observe that if $g:C'\longrightarrow T$ is a monomorphism in $\Ccal$ with $C'\in\Ccal$ and $T\in\Tcal_1$, then, by lemma \ref{epi}, $C'$ lies in $\Tcal$ and $g$ is a monomorphism in $\Bcal$.  Therefore, since $\Tcal_1$ is hereditary in $\Bcal$, $C'$ lies in $\Tcal_1$. We furthermore observe that $\Tcal_1$ is compactly generated in $\Dcal^b(\Acal)$ as a subcategory of $\Ccal$. Indeed, for $X$ in $\Tcal_1$, consider a family $(X_j)_{j\in J}$ of subobjects of $X$ in $\Bcal$, compact in $\Dcal^b(\Acal)$,  such that $X=\varinjlim_{j\in J} X_j$. Since $\Tcal_1$ is a hereditary torsion class in $\Bcal$, each $X_j$ lies in $\Tcal_1$ and it is a subobject of $X$ in $\Ccal$. Now, $X$ lies in $\Ccal$ and so does the coproduct of the family $(X_j)_{j\in J}$ (more precisely it lies in $\Tcal_1$). The colimit of this family in $\Ccal$ is the cokernel of an endomorphism of the coproduct (see, for example, \cite{St}, IV.8.4) and, thus, it is still $X$, as wanted.

Note that, by lemma \ref{cocomplete heart}, $\Ccal$ is cocomplete. To complete the proof that $\Tcal_1$ is a hereditary torsion class in $\Ccal$ using lemma \ref{tor char}, we just need to show that $\Ccal$ is $\Xcal$-well-powered, where $\Xcal$ is the subcategory of $\Tcal_1$ formed by those objects which are compact in $\Dcal^b(\Acal)$. This follows from the fact that $\Tcal_1$ is a subcategory of $\Ccal$ compactly generated in $\Dcal^b(\Acal)$ and from our assumption that the compact objects of $\Dcal^b(\Acal)$ lying in $\Tcal$ form a set (and, thus, so do the ones lying in $\Tcal_1$). It is then clear that, for any object $Y$ in $\Ccal$, the family of subobjects of $Y$ lying in $\Xcal$ lies in the product of the sets $Hom_\Ccal(X,Y)$, where $X$ runs over the set $\Xcal$.

Conversely, suppose $\Tcal_1$ is a torsion class in $\Ccal$. As a subcategory of $\Bcal$ it is again obviously closed under extensions and coproducts. Suppose that $f:T\longrightarrow B$  is an epimorphism in $\Bcal$ with $T\in\Tcal_1\subseteq \Tcal$ and $B\in\Bcal$. Then, clearly $B\in \Tcal$ and, by lemma \ref{epi}, $f$ is an epimorphism in $\Ccal$. Hence, since $\Tcal_1$ is a torsion class in $\Ccal$, $B$ must lie in $\Tcal_1$. Moreover, observe that if $g:B'\longrightarrow T$ is a monomorphism in $\Bcal$ with $B'\in\Bcal$ and $T\in\Tcal_1$, then, clearly $B'\in\Tcal$ and, by lemma \ref{epi}, $g$ is a monomorphism in $\Ccal$.  Therefore, since $\Tcal_1$ is hereditary in $\Ccal$, $B'$ lies in $\Tcal_1$. Finally, since $\Bcal$ is cocomplete and $\Tcal_1$ is $\Xcal$-well-powered, where $\Xcal$ is the set of compact objects of $\Dcal^b(\Acal)$ lying in $\Tcal_1$, lemma \ref{tor char} concludes the proof.
\end{proof}

We need one more simple but useful lemma (in light of remark \ref{aisle contained}) about the relation between truncations of t-structures whose aisles are related by inclusion.

\begin{lemma}\label{two t-s}
Suppose $(\mathcal{D}_A^{\leq 0},\mathcal{D}_A^{\geq 0})$ and $(\mathcal{D}_B^{\leq 0},\mathcal{D}_B^{\geq 0})$ are two t-structures in a triangulated category $\mathcal{D}$ with truncation functors $\tau^{\leq0}_A$ and $\tau^{\leq0}_B$, respectively. If $\mathcal{D}_A^{\leq 0}\subset\mathcal{D}_B^{\leq 0}$, then for all $X\in\mathcal{D}$ there is a triangle:
\begin{equation}\label{two trunc}
\tau^{\leq 0}_A(X)\longrightarrow \tau^{\leq 0}_B(X) \longrightarrow Y \longrightarrow \tau^{\leq 0}_A(X)[1]
\end{equation}
such that $Y\in D_A^{\geq 1}\cap D_B^{\leq 0}$.
\end{lemma}
\begin{proof}
First note that, since $\mathcal{D}_A^{\leq 0}\subset\mathcal{D}_B^{\leq 0}$, we have $\mathcal{D}_B^{\geq 1}\subset\mathcal{D}_A^{\geq 1}$. The triangle
\begin{equation}\nonumber
\tau_B^{\leq 0}(X)\longrightarrow  X \longrightarrow \tau_B^{\geq 1}(X) \longrightarrow \tau_B^{\leq 0}(X)[1]
\end{equation}
then shows that the natural map $\tau_A^{\leq 0}(X)\longrightarrow  X$ must factor through $\tau_B^{\leq 0}(X)$ (since Hom$(\tau_A^{\leq 0}(X),\tau_B^{\geq 1}(X))=0$). Let $Y$ be defined by the following triangle
\begin{equation}\nonumber
\tau_A^{\leq 0}(X)\longrightarrow  \tau_B^{\leq 0}(X) \longrightarrow Y \longrightarrow \tau_A^{\leq 0}(X)[1].
\end{equation}
Since aisles are closed under taking cones and $\tau_A^{\leq 0}(X)\in D_B^{\leq 0}$, we have that $Y\in D_B^{\leq 0}$. We want to prove $Y\in D_A^{\geq 1}$. Consider the diagram
\begin{equation}\nonumber
\xymatrix{\tau_A^{\leq 0}(X)\ar[r]\ar[d]&\tau_B^{\leq 0}(X)\ar[r]\ar[d]&Y\ar[r]& \tau_A^{\leq 0}(X)[1]\ar[d]\\ \tau_A^{\leq 0}(X)\ar[r]\ar[d]&X\ar[r]\ar[d]&\tau_A^{\geq 1}(X)\ar[r]& \tau_A^{\leq 0}(X)[1]\ar[d]\\ \tau_B^{\leq 0}(X)\ar[r]&X\ar[r]&\tau_B^{\geq 1}(X)\ar[r]& \tau_B^{\leq 0}(X)[1]}
\end{equation}
where rows are triangles and the squares commute by the observation above. Then, the octahedral axiom gives us a new triangle
\begin{equation}\nonumber
Y\longrightarrow  \tau_A^{\geq 1}(X) \longrightarrow \tau_B^{\geq 1}(X) \longrightarrow Y[1].
\end{equation}
Since $\tau_B^{\geq 1}(X)\in D_A^{\geq 1}$, so is $\tau_B^{\geq 1}(X)[-1]$. By the long exact sequence of cohomology induced from this triangle, it is easy to see that this shows that $Y\in D_A^{\geq 1}$.
\end{proof}

We say that a heart $\Bcal$ is obtained by \textit{iterated HRS-tilts} in $\Dcal^b(\Acal)$ if there is a finite sequence of hearts $\Acal=\Acal_0, \Acal_1,...,\Acal_n=\Bcal$ and a sequence of subcategories $\Tcal_0,...,\Tcal_{n-1}$ such that $\Tcal_i$ is a torsion class in $\Acal_i$ and $\Acal_{i+1}$ is the HRS-tilts of $\Acal_i$ with respect to the torsion theory given by $\Tcal_i$, for all $0\leq i\leq n-1$. We now prove the main result of this section. 

\begin{theorem}\label{main it}
Let $\Acal$ be an AB4 abelian category and $S = \left\{\Tcal_a,\Tcal_{a+1}...,\Tcal_{a+n-1}\right\}$ a set of hereditary torsion classes of $\Acal$, compactly generated in $\Dcal^b(\Acal)$, such that
$$\Tcal_a\supseteq \Tcal_{a+1}\supseteq \Tcal_{a+2}\supseteq ...\supseteq \Tcal_{a+n-1}=0$$
and such that the compact objects of $\Dcal^b(\Acal)$ lying in $\Tcal_a$ form a set. Then, the full subcategory given by 
\begin{equation}\nonumber
\Dcal^{S,\leq 0} = \left\{ X^{\bullet}\in \Dcal^b(\mathcal{A}): \ H_0^i(X^{\bullet})\in \Tcal_j,\ \forall i>j\right\}
\end{equation}
is the aisle of a uniformly bounded t-structure in $\Dcal^b(\Acal)$ with a cocomplete heart $\Bcal$ and it is obtained by iterated HRS-tilts with respect to the sequence $S$.
\end{theorem}
\begin{proof}
Without loss of generality, we assume that $a=-n+1$. We use induction on the number $n$ of elements of $S$ to show that $\Dcal^{S,\leq 0}$ can be obtained by iterated HRS-tilts with respect to a sequence of torsion classes given by the following chain
\begin{equation}\nonumber
\Tcal_{-n+1}\supseteq \Tcal_{-n+2} \supseteq ... \supseteq \Tcal_{-1}\supseteq \Tcal_0=0 
\end{equation}

Suppose $n=1$, i.e., $S=\left\{\Tcal_0=0\right\}$. Then, we have
\begin{equation}\nonumber
\Dcal^{S,\leq 0} = \left\{ X\in \Dcal^b(\mathcal{A}): H^i_0(X)=0, \forall i>0\right\} = \Dcal^{\leq 0}_0,
\end{equation}
the standard aisle in $\Dcal^b(\Acal)$, which clearly satisfies all the desired properties.

Suppose the result is valid for sequences $S$ of $n$ torsion classes satisfying the assumptions of the theorem. Let $S$ be a sequence with $n+1$ hereditary torsion classes of $\Acal$ which are compactly generated in $\Dcal^b(\Acal)$, $$S=\left\{\Tcal_{-n},\Tcal_{-n+1},...,\Tcal_{0}\right\}\ \ \ {\rm such\ that}\ \ \ \Tcal_{-n}\supseteq \Tcal_{-n+1}\supseteq ...\supseteq \Tcal_{-1}\supseteq \Tcal_0=0.$$ 
Let us consider the sequence $$\overline{S}=\left\{\overline{T}_{-n+1},\overline{T}_{-n+2},...,\overline{T}_{0}\right\} \ \ \ {\rm where}\ \ \ \overline{T}_i=\Tcal_{i-1}, \forall i<0, \overline{T}_0=0.$$  Clearly, $\overline{S}$ is also a decreasing chain of hereditary torsion classes of $\mathcal{A}$, compactly generated in $\Dcal^b(\Acal)$. We fall into the case of $n$ torsion classes and by the induction hypothesis we have an associated uniformly bounded t-structure with a cocomplete heart given by the aisle obtained by iterated HRS-tilts with respect to $\overline{S}$,
\[\Dcal^{\overline{S},\leq 0}=\{X\in\Dcal^b(\Acal): H_0^i(X)\in \overline{T}_j, \forall i>j\}=\] 
\[\ \ \ \ \ \ \ \ =\{X\in\Dcal^b(\Acal): X\in\Dcal_0^{\leq 0}, H_0^i(X)\in \Tcal_{j-1}, \forall i>j\}.\]
We denote the corresponding heart by $\mathcal{B}$ and associated cohomological functor by $H_{\overline{S}}^0:= \tau_{\overline{S}}^{\geq 0}\tau_{\overline{S}}^{\leq 0}$, where the $\tau_{\overline{S}}$'s are the associated truncation functors.
Observe now that $\Tcal_{-1}$ is a subcategory of $\mathcal{B}$. This follows from the fact that it is contained on every torsion class in $\overline{S}$ and that $\Bcal$ is obtained by iterated HRS-tilts. By applying iteratively lemma \ref{new torsion class}, we can also conclude that $\Tcal_{-1}$ is a hereditary torsion class in $\Bcal$, compactly generated in $\Dcal^b(\Acal)$.
By lemma \ref{cocomplete heart} we get that the HRS-tilt of $\Bcal$ with respect to $\Tcal_{-1}$ yields a uniformly bounded t-structure with a cocomplete heart. It remains to show that the aisle of the HRS-tilt can be expressed in terms of the torsion classes in $S$ as wanted. That follows from the following lemma. 

\begin{lemma} An object $X$ lies in $\Dcal^{S,\leq 0}$ if and only if $H_{\overline{S}}^0(X)$ lies in $\Tcal_{-1}$ and $H_{\overline{S}}^i(X)=0$, for all $i>0$.
\end{lemma}
\begin{proof}
Suppose that $X$ lies in $\Dcal^{S,\leq 0}$. It is clear from the definition of $\overline{S}$ that $\Dcal^{S,\leq 0}\subseteq \Dcal^{\overline{S},\leq 0}$, thus proving that $X$ lies in $\Dcal^{\overline{S},\leq 0}$, which is equivalent to the condition that $H_{\overline{S}}^i(X)=0$, for all $i>0$.
Note that $H^0_{\overline{S}}(X)$ fits in the triangle
\begin{equation}\nonumber
\tau_{\overline{S}}^{\leq -1}(X)\longrightarrow \tau_{\overline{S}}^{\leq 0}(X)\longrightarrow H^0_{\overline{S}}(X)\longrightarrow \tau_{\overline{S}}^{\leq -1}(X)[1],
\end{equation} 
which, again due to the fact that $\Dcal^{S,\leq 0}\subseteq \Dcal^{\overline{S},\leq 0}$, amounts to the triangle
\begin{equation}\label{triangle1}
\tau_{\overline{S}}^{\leq -1}(X)\longrightarrow X \longrightarrow H^0_{\overline{S}}(X)\longrightarrow \tau_{\overline{S}}^{\leq -1}(X)[1].
\end{equation}
Now, lemma \ref{two t-s} applied to $\Dcal^{\overline{S},\leq -1}\subset \Dcal^{\leq -1}_0$ (see remark \ref{aisle contained}) shows that there is $Y\in \Dcal^{\overline{S},\geq 0}\cap \Dcal_0^{\leq -1}$ and a triangle
\begin{equation}\nonumber
\tau^{\leq -1}_{\overline{S}}(X)\longrightarrow \tau_0^{\leq -1}(X) \longrightarrow Y \longrightarrow \tau^{\leq 0}_{\overline{S}}(X)[1].
\end{equation}
Since $X$ lies in $\Dcal^{S,\leq 0}$, we have that $\tau_0^{\leq -1}(X)$ lies in $\Dcal^{\overline{S},\leq -1}$, by construction of $\overline{S}$ and, thus, we get $Y=0$ and $\tau_{\overline{S}}^{\leq -1}(X) \cong \tau_0^{\leq -1}(X)$. Since two of the vertices of a triangle determine the third one up to isomorphism, the triangle \ref{triangle1} shows that $H^0_{\overline{S}}(X)= H^0_0(X)$, which, by definition of $\Dcal^{S,\leq 0}$, tells us that $H^0_{\overline{S}}(X)$ lies in $\Tcal_{-1}$.

Conversely, suppose $X\in \Dcal^{\overline{S},\leq 0}$ and $H^0_{\overline{S}}(X)\in \mathcal{T}_{-1}$. As before, we have a triangle 
\begin{equation}\nonumber
\tau_{\overline{S}}^{\leq -1}(X)\longrightarrow X \longrightarrow H^0_{\overline{S}}(X)\longrightarrow \tau_{\overline{S}}^{\leq -1}(X)[1].
\end{equation}
We now apply to it the standard cohomology functor, getting a long exact sequence of objects in $\Acal$. On one hand, since $\Dcal^{\leq -1}_{\overline{S}}$ is contained in $\Dcal_0^{\leq -1}$, we see that $H^i_0(\tau_{\overline{S}}^{\leq -1}(X))=0$ for all $i>-1$. This shows, in particular, that 
\begin{equation}\label{isomorphism1}
H_0^0(X)\cong H^0_0(H^0_{\overline{S}}(X))
\end{equation}
and, therefore, $H_0^0(X)$ lies in $\Tcal_{-1}$ (hence, it lies in $\Tcal_j$ for all $j<0$). Let us now look at the negative standard cohomologies of $X$. Since, by hypothesis, $H^i_0(H^0_{\overline{S}}(X))=0$ for all $i\neq 0$, the same long exact sequence also yields the following isomorphisms
\begin{equation}\label{equation negative coh}
H^{i}_0(\tau_{\overline{S}}^{\leq -1}(X))\cong H^{i}_0(X), \forall i<0.
\end{equation}
Since $\Dcal^{\overline{S},\leq -1} = \Dcal^{\overline{S},\leq 0}[1]$, we have
\begin{equation}\nonumber
\Dcal^{\overline{S},\leq -1} = \{Y\in\Dcal^b(\Acal): H^i_0(Y)\in \overline{\Tcal}_j, \forall i>j-1\}
\end{equation}
and, thus, the isomorphisms (\ref{equation negative coh}) show that, in particular,
$$H^{i}_0(X)\cong H^{i}_0(\tau_{\overline{S}}^{\leq -1}(X))\in \overline{\Tcal}_{i}=\Tcal_{i-1}, \forall i<0.$$   This proves that, for all $0>i>j$, $H^i_0(X)$ lies in $\Tcal_j$, since the torsion classes form a chain. This fact, together with the isomorphism (\ref{isomorphism1}) proves that $X$ lies in $\Dcal^{S,\leq 0}$.

\end{proof}

$\Dcal^{S,\leq 0}$ can thus be obtained as the aisle of the HRS-tilt of the heart $\mathcal{B}$, defined earlier in this proof, with respect to the torsion theory whose torsion class is $\mathcal{T}_{-1}$, hence finishing the proof.
\end{proof}
\begin{remark}\label{good cats}
In the following sections of the paper we will use this construction in categories of the form $\Dcal^b(Tails(R))$, where $R$ is a positively graded, connected, noetherian $\mathbb{K}$-algebra generated in degree 1 such that one of the following holds:
\begin{itemize}
\item $R$ is commutative (and, thus, $\Dcal^b(Tails(R))\cong \Dcal^b(Qcoh(Proj(R)))$); 
\item $R$ is a 3-dimensional Artin-Schelter regular $\mathbb{K}$-algebra, which is finitely generated as a module over its centre.
\end{itemize}
The technical assumptions on the torsion classes will be satisfied in these contexts.
\begin{enumerate}
\item The category $Tails(R)$ is AB4 (it is, in fact, a Grothendieck category), where every object is a colimit of its subobjects lying in $tails(R)$. The subcategory of compact objects of $\Dcal^b(Tails(R))$ is equivalent to $\Dcal^b(tails(R))$  (\cite{BvdB}, lemmas 4.3.2 and 4.3.3) and, thus, every hereditary torsion class in $Tails(R)$ is compactly generated in $\Dcal^b(Tails(R))$. 

\item The functors $\pi$ and $\Gamma_*$ induce equivalences between the additive categories of torsion-free injective objects in $Gr(R)$ and injective objects in $Tails(R)$ (see \cite{Gab}, corollaries 2 and 3, pp.375, and also lemma \ref{hom tails 0} below). In both cases considered, injective objects in $Gr(R)$ (or in $Tails(R)$) are direct sums of indecomposable injective objects (\cite{SM}, see also proposition \ref{dec} below), which form a set $Inj(Gr(R))$ (or $Inj(Tails(R))$)  parametrised (up to isomorphism and shifts) by homogeneous prime ideals (\cite{Matlis},\cite{NVO},\cite{VOV}). Any torsion theory $(\Tcal,\Fcal)$ considered is cogenerated by a subset $X$ of $Inj(Tails(R))$ and, thus, injective envelopes of objects in $\Tcal$ are direct sums of objects in $Inj(Tails(R))\setminus X$. In our cases, the adjoint pair $(\pi,\Gamma_*)$ restricts to an adjoint pair between $gr(R)$ and $tails(R)$ (see \cite{AZ}) and, thus, the objects of $\Tcal$ which are compact in $\Dcal^b(Tails(R))$ have injective envelopes given by finite direct sums of objects in $Inj(Tails(R))\setminus X$, thus forming a set.
\end{enumerate}

\end{remark}
\end{section}
\begin{section}{Perverse coherent t-structures through torsion theories}

In this section, we will prove our main theorem. We start by fixing some notation. Let $X$ be a smooth projective scheme over an algebraically closed field $\mathbb{K}$ such that its homogeneous coordinate ring $R=\Gamma_*(X)$ is a commutative noetherian positively graded $\mathbb{K}$-algebra generated in degree 1.  We denote by $\pi$ the projection functor from $Gr(R)$ to its quotient $Tails(R)$ (and the corresponding restriction to $gr(R)$). It has a right adjoint given by
\begin{equation}\nonumber
\Gamma_*(\pi M)=\bigoplus\limits_{i\in \mathbb{Z}} {\rm Hom}_{Tails(R)}(\pi R,\pi M(i)).
\end{equation}
For more details on the formalism of these quotient categories check \cite{AZ}, for instance.

Let $p:X^{top} \longrightarrow \mathbb{Z}$ be a perversity as defined in the introduction. Suppose that the perversity has $n$ values and that, without loss of generality, the maximal value of the perversity is zero. Let $I_x$ be the homogeneous ideal of functions vanishing at an element $x$ of $X^{top}$ and, for $i$ in $Im(p)$, define $$E_i=\prod\limits_{\left\{x\in X^{top}:p(x)\leq i\right\}}E^g(R/I_x).$$

\begin{lemma}\label{hom tails 0}
Let $R$ be a ring and $A$ and $B$ graded $R$-modules. If $B$ is torsion-free and injective, then we have an isomorphism Hom$_{Gr(R)}(A,B)\cong$ Hom$_{Tails(R)}(\pi A, \pi B)$.
\end{lemma}
\begin{proof}
This follows from \cite{Gab} (lemme 1, proposition 3, pp. 370--371). Since $\Gamma_*$ is right adjoint to $\pi$, we have Hom$_{Tails(R)}(\pi A, \pi B)\cong$ Hom$_{Gr(R)}(A,\Gamma_*\pi B)$. The unit of the adjunction $\phi:B\rightarrow \Gamma_*\pi B$ has a torsion kernel and a torsion cokernel and, since $B$ is torsion-free, $\phi$ must be and injective map. Since $B$ is an injective object, the short exact sequence induced by $\phi$ splits and, since $\Gamma_*\pi B$ is torsion-free, we conclude that $\phi$ is an isomorphism.  
\end{proof}

Recall that an injective object $I$ of an abelian category $\mathcal{A}$ cogenerates $\Acal$ if, for any $X$ in $\mathcal{A}$, Hom$_\Acal(X,I)=0$ implies $X=0$, i.e., the associated torsion class, $\Tcal_I$, is zero (since $\Tcal_I$ is the kernel of the functor Hom$_{\Acal}(-,I)$). 
\begin{remark}\label{torsion-free}
Note that given $R$ positively graded noetherian connected $\mathbb{K}$-algebra, $R/P$ is torsion-free for any homogeneous prime ideal $P$ not equal to the irrelevant ideal. Indeed, for $x\notin P$, if $xR_{\geq n}=0$, then $(RxR)(R_{\geq n})\subseteq P$ and hence, by \ref{rem}, $RxR\subseteq P$, which yields a contradiction.
\end{remark}

Note that all modules $E_i$ are torsion-free by the remark above. Indeed, since the torsion-free class of a hereditary torsion theory is closed under taking injective envelopes, $E^g(R/P)$ is torsion-free. Also, $\pi E_i$ is injective in $Tails(R)$, for all $i$, since $\pi$ is essentially surjective and $\Gamma_*$ is left exact. 

\begin{corollary}
The object $\pi E_0$ cogenerates $Tails(R)$, where $R=\Gamma_*(X)$ is as above.
\end{corollary}

\begin{proof}
Suppose that $M$ is not torsion, i.e., that there is an element $m\in h(M)$ such that $Ann(m)\neq R_{\geq n}$ for any $n>1$. We prove that $Ann(m)$ is contained in a homogeneous prime ideal. It is clear that, since $m$ is not torsion, the radical of $Ann(m)$, which we denote by $\sqrt{Ann(m)}$, is not the augmentation ideal $R_+$. Thus we can choose $f\in R_1$ such that $f\notin\sqrt{Ann(m)}$. Applying Zorn's lemma to the set $S=\left\{J\supset Ann(m)\  {\rm homogeneous}: f\notin \sqrt{J}\right\}$ (which is nonempty since $Ann(m)\in S$) we get a maximal element - call it $P$. We prove that $P$ is prime. In fact, for $a, b\in h(R)$, if $ab\in P$ and $a\notin P$, then there is an integer $l$ such that $f^l\in aR+P$ (since $P$ is maximal in $S$). If there is an integer $s$ such that $f^s\in bR+P$, then $f^{l+s}\in (aR+P)(bR+P)\subset P$, a contradiction. Hence $b\in P$. This proves that $P$ is a homogeneous gr-prime ideal. 

To complete the proof we need the lemma below. Recall that in noncommutative ring theory, primality of an ideal $P$ is defined in terms of products of ideals, i.e., if $IJ\subset P$ for some ideals $I$ and $J$, then $I\subset P$ or $J\subset P$. If this property holds at the level of elements (i.e., if $ab\in P$ for some elements $a$, $b$ of the ring, then $a\in P$ or $b\in P$) then we say $P$ is strongly prime. There are obvious graded counterparts of these notions and the following property holds.
\begin{lemma}[Nastasescu, Van Oystaeyen, \cite{NVO2}]
For a $\mathbb{Z}$-graded ring, a homogeneous ideal is gr-strongly prime if and only if it is strongly prime.
\end{lemma}
Since $R$ is commutative, the notions of prime and strongly prime coincide. Hence, $P$ is prime.

Note now that there is a graded isomorphism from $R/Ann(m)(deg(m))$ to $mR$ and thus a graded injection from $R/Ann(m)(deg(m))$ to $M$. Since $Ann(m)$ is contained in a homogeneous prime ideal $P$, $R/Ann(m)(deg(m))$ maps nontrivially to $E^g(R/P)(deg(m))$ and thus so does $M$. Since $R$ satisfies the hypothesis of lemma \ref{cogen is rigid}, one has that $M$ maps nontrivially to $E^g(R/P)$ and thus, by the previous lemma, Hom$_{Tails(R)}(\pi M, \pi E^g(R/P))\neq 0$. 
\end{proof}

Before stating the main theorem, we need to prove the following useful lemma.
\begin{lemma}\label{local ring}
Suppose $R$ is a commutative local ring with maximal ideal $m$. Given $X^{\bullet}$ a bounded complex of finitely generated free R-modules, define $Y^{\bullet}$ to be the complex $R/m\otimes_R X^{\bullet}$. If, for some fixed integer $\alpha$, $H_0^j(Y^{\bullet})=0$ for all $j\geq \alpha$, then $H_0^j(X^{\bullet})=0$ for all $j\geq \alpha$. 
\end{lemma}
\begin{proof}
Without loss of generality, let us assume that $H^j_0(Y^{\bullet})=0$ for all $j\geq 1$, i.e., $\alpha=1$. Since $X^\bullet$ is a bounded complex, let $p\in\Zbb$ be its right bound, i.e., $X^{k}=0$ (and hence $Y^k=0$) for all $k> p$. 

If $p<1$ then the result trivially follows. 
Suppose that $p\geq 1$. We will first show that the cohomology $H^p_0(X^\bullet)$ vanishes. Consider the exact sequence
\begin{equation}\nonumber
X^{p-1}\longrightarrow X^p \longrightarrow coker(d_X^{p-1}) \longrightarrow 0
\end{equation}
and apply to it the functor $R/m \otimes_R -$, thus getting another exact sequence
\begin{equation}\nonumber
Y^{p-1}\longrightarrow Y^{p} \longrightarrow R/m\otimes_R coker(d_X^{p-1}) \longrightarrow 0,
\end{equation}
since $R/m\otimes_R -$ is right exact. By definition of $Y^{\bullet}$, the first map of the sequence is the differential $d^{p-1}_Y$. Since $1 \leq p$, $H_0^p(Y^{\bullet})=0$, thus proving that $d^{p-1}_Y$ is surjective ($Y^{p+1}=0$ by definition). Therefore $R/m\otimes_R coker(d_X^{p-1})=0$ which, by Nakayama's lemma (since $R$ is local and $coker(d_X^{p-1})$ is a finitely generated $R$-module), implies that $coker(d_X^{p-1})=0$. Hence $d_X^{p-1}$ is surjective, thus proving that $H_0^p(X^{\bullet})=0$. 

We now prove our result by induction on $p\geq 1$. If $p=1$, the previous paragraph shows that $H_0^1(X^\bullet)=0$ and the result follows. Suppose now that the result is valid for all complexes of free $R$-modules $X^\bullet$ with right bound $p\geq 1$ and let $X^\bullet$ be a complex of free $R$-modules with right bound $p+1$. The previous paragraph shows that $H_0^{p+1}(X^\bullet)=0$. Since $X^{p+2}=0$, there is a short exact sequence
\begin{equation}\nonumber
0 \longrightarrow Ker(d_X^{p}) \longrightarrow X^{p}\longrightarrow X^{p+1} \longrightarrow 0
\end{equation}
which splits since $X^{p+1}$ is free. Thus, $Ker(d_X^{p-1})$ is a summand of the free module $X^{p-1}$, i.e., it is a projective module. However, it is well-known (Kaplansky's theorem) that projective modules over local rings are free and, hence, the complex
\begin{equation}\nonumber
\tilde{X}^\bullet:= \xymatrix{...\ar[r]& X^{p-2}\ar[r]^{d_X^{p-2}}&X^{p-1}\ar[r]^{d_X^{p-1}}&Ker(d_X^{p})\ar[r]&0\ar[r] &...}.
\end{equation}
is a complex of free $R$-modules which is quasi-isomorphic to $X^\bullet$. Since $X^\bullet$ is a complex of free $R$-modules, its tensor product with $R/m$ can be regarded in $\Dcal(Mod(R))$ as the derived tensor product 
$$Y^\bullet\cong R/m\otimes^\Lbb_R X^\bullet\cong R/m\otimes^\Lbb_R \tilde{X}^\bullet$$ 
in $\Dcal(Mod(R))$, where the second isomorphism holds since $X^\bullet$ and $\tilde{X}^\bullet$ are isomorphic in $\Dcal(Mod(R))$. Since $\tilde{X}^\bullet$ is also a complex of free $R$-modules, we have that $Y^\bullet$ is isomorphic to $R/m\otimes_R \tilde{X}^{\bullet}$ in $\Dcal(Mod(R))$, i.e., they are are quasi-isomorphic complexes and, thus $H_0^{i}(R/m\otimes_R \tilde{X}^\bullet)=0$ for all $i\geq 1$. 
Now, the induction hypothesis holds for $\tilde{X}^\bullet$ and, thus, $H^i_0(\tilde{X}^\bullet)=0$ for all $i\geq 1$. Since $\tilde{X}^\bullet$ is quasi-isomorphic to $X^\bullet$, we also have that $H^i_0(X^\bullet)=0$ for all $i\geq 1$, thus finishing the proof.

\end{proof}

Finally we prove our main theorem. It gives us a description of the aisle of a perverse coherent t-structure $\Dcal^{p,\leq 0}$ is terms of torsion classes. 
\begin{theorem}\label{main perverse}
Let $X$ be a smooth projective scheme over $\mathbb{K}$, $R=\Gamma_*(X)$ its homogeneous coordinate ring and $p$ a perversity on $X$. Suppose that $R$ is a commutative connected, noetherian, positively graded $\mathbb{K}$-algebra generated in degree 1. Let $\Tcal_i$ denote the torsion class cogenerated  in $Tails(R)$ by $\pi E_i$, where $$E_i=\prod\limits_{\left\{x\in X^{top}:p(x)\leq i\right\}}E^g(R/I_x),$$ with $I_x$ standing for the defining ideal of $x\in X^{top}$ in $R$. Then we have:
\begin{equation}\nonumber
\Dcal^{p,\leq 0} = \left\{ F^{\bullet}\in \Dcal^b(Tails(R)): H_0^i(F^{\bullet})\in \Tcal_j, \forall i>j\right\}\cap \Dcal^b(tails(R)).
\end{equation}
\end{theorem}
\begin{proof}
Let $S$ be the set of torsion classes $\Tcal_i$. By remark \ref{good cats} and theorem \ref{main it}, $\Dcal^{S,\leq 0}$ is an aisle in $\Dcal^b(Tails(R))$. We will prove that the subcategories $\Dcal^{p,\leq 0}$ and $\Dcal^{S,\leq 0}\cap\Dcal^b(tails(R))$ coincide. We denote by $\widehat{\Tcal_i}$ the torsion theory cogenerated by $E_i$ in $Gr(R)$. We start by rewriting the conditions defining the aisle $\Dcal^{S,\leq 0}$. By definition, we have 
\begin{equation}\nonumber
\Dcal^{S,\leq 0}=\left\{F^{\bullet}\in \Dcal^{b}(Tails(R)):\ H^j_0(F^{\bullet})\in \Tcal_k,\  \forall j>k\right\}
\end{equation}
and, given that $E_k$, for all $k$, is torsion-free injective, by lemma \ref{hom tails 0} we have
\begin{equation}\nonumber
\Dcal^{S,\leq 0}=\left\{F^{\bullet}\in \Dcal^{b}(Tails(R)):\ \Gamma_*(H^j_0(F^{\bullet}))\in \widehat{\Tcal_k},\  \forall j>k\right\}=
\end{equation}
\begin{equation}\nonumber
\left\{F^{\bullet}\in \Dcal^{b}(Qcoh(X)): \forall x \in X^{top},{\rm Hom}_{Gr(R)}(\Gamma_*(H^j_0(F^{\bullet})), E^g(R/I_x))=0,\forall j>p(x)\right\}.
\end{equation}
We intersect with $\Dcal^b(coh(X))$ to pass from a quasi-coherent to a coherent setting. For simplicity, define $\Dcal^{s,\leq 0}:=\Dcal^{S,\leq 0}\cap \Dcal^b(coh(X))$. By corollary \ref{corollary deg zero}, we get
\begin{equation}\nonumber
\Dcal^{s,\leq 0}=\left\{F^{\bullet}\in \Dcal^{b}(coh(X)): \forall x \in X^{top},\Gamma_*(H^j_0(F^{\bullet}))_{(x)}=0,\forall j>p(x)\right\}
\end{equation}
where $\Gamma_*(H^j_0(F^{\bullet}))_{(x)}$ is the degree zero part of the localisation of $\Gamma_*(H^j_0(F^{\bullet}))$ at the prime ideal $I_x$, which is the same as stalk at $x$ of the sheaf $H^j_0(F^{\bullet})$. Since taking stalks is an exact functor (thus t-exact for the standard t-structure and therefore commuting with cohomology functors) we get
\begin{equation}\nonumber
\Dcal^{s,\leq 0}=\left\{F^{\bullet}\in \Dcal^{b}(coh(X)): \forall x \in X^{top},\  H^j_0(F^{\bullet}_{x})=0,\  \forall j>p(x)\right\}.
\end{equation}

On the other hand, recall that
\begin{equation}\nonumber
\Dcal^{p,\leq 0}=\left\{F^{\bullet}\in \Dcal^{b}(coh(X)): \forall x \in X^{top},\  Li_{x}^{*}(F^{\bullet}) \in \Dcal_0^{\leq p(x)}(O_{\{x\}}\mbox{-}mod)\right\},
\end{equation}
which is clearly the same as 
\begin{equation}\nonumber
\left\{F^{\bullet}\in \Dcal^{b}(coh(X)): \forall x \in X^{top},\  H^j_0(Li_{x}^{*}(F^{\bullet}))=0,\ \forall j>p(x)\right\}.
\end{equation}
Therefore, it suffices to prove that $H^j_0(F^{\bullet}_{x})=0$ for all $j>p(x)$ if and only if $L_ji_{x}^{*}(F^{\bullet})=H_0^j(Li_{x}^{*}(F^{\bullet}))=0$ for all $j>p(x)$.

Let $F^{\bullet}\in \Dcal^b(coh(X))$ such that $H^j_0(F^{\bullet}_{x})=0$ for all $j>p(x)$. By definition of the pullback functor ($i_x^{*}(V)=V_x\otimes_{O_{X,x}} k(x)$ for any coherent sheaf $V$, where $k(x)=O_{X,x}/m_{X,x}$, with $m_{X,x}$ being the maximal ideal of the local ring $O_{X,x}$, is the residue field at the point $x$), there is a spectral sequence of Grothendieck type of the following form:
\begin{equation}\nonumber
E_2^{ab}=Tor_{a}^{O_{X,x}}(k(x),H_0^b(F^{\bullet}_x)) \Longrightarrow L_{a+b}i_x^*(F^{\bullet}).
\end{equation}
Our hypothesis shows that $E_2^{ab}=0$ for all  $b>p(x)$ (and, of course, by definition of $Tor$, also for all $a<0$). Thus $E_{\infty}^{ab}=0$ for all $a<0$ or $b>p(x)$. Let $\mathbb{F}^{i}$ denote the $i$-th part of the decreasing filtration assumed to exist (by definition of convergent spectral sequence) on the limit object $\Omega^{a+b}:=L_{a+b}i_x^*(F^{\bullet})$. Then, for $q>p(x)$,
\begin{equation}\nonumber
... = \mathbb{F}^{-2}\Omega^{-2+(q+2)} = \mathbb{F}^{-1}\Omega^{-1+(q+1)} = \mathbb{F}^0\Omega^q = \mathbb{F}^1\Omega^q 
\end{equation}
and thus they are all equal to zero, proving that $\Omega^q=L_qi_x^*(F^{\bullet})=0$ for all $q>p(x)$.

Conversely, suppose we have $F^{\bullet}$ such that $L_ji_{x}^{*}(F^{\bullet})=0$ for all $j>p(x)$. Since $X$ is smooth, let $G^{\bullet}$ be a bounded complex of locally free sheaves such that $G^{\bullet}$ is quasi-isomorphic to $F^{\bullet}$ (thus isomorphic in the derived category) - check, for example, \cite{Huy}, Proposition 3.26. Then $L_ji_{x}^{*}(F^{\bullet})=0$ means that $H^j_0((i_{x}^{*}G)^{\bullet})=0$, where $(i_{x}^{*}G)^{\bullet}$ denotes the complex resulting from applying $i_x^*$ componentwise to $G^{\bullet}$. Take now $X^{\bullet}=G^{\bullet}_x$ and $Y^{\bullet}=(i_{x}^{*}G)^{\bullet}$ and recall that $G^{\bullet}_x$ is a complex of free modules over the local ring $O_{X,x}$. This leaves us in the context of \ref{local ring}, thus proving that $H^j_0(G^\bullet_x)=H^j_0(G^{\bullet})_x=0$ for all $j>p(x)$. Finally we have $H^j_0(F^{\bullet}_x)=H^j_0(F^{\bullet})_x=H^j_0(G^{\bullet})_x = 0$ for all $j>p(x)$, hence finishing the proof.
\end{proof}
\begin{remark}\label{restriction} 
This theorem compares two subcategories of $\Dcal^b(coh(X))$, showing that they coincide. It contains no proof that either of the subcategories involved are aisles of t-structures. However, by proving that the intersection $\Dcal^{S,\leq 0}\cap \Dcal^b(coh(X))$, for $S$ defined as above, coincides with the aisle $\Dcal^{p,\leq 0}$ constructed in \cite{Be}, we do show that, under the assumptions of the theorem, $\Dcal^{S,\leq 0}$ (which is an aisle by section 3) in $\Dcal^b(Qcoh(X))$ restricts well to an aisle in $\Dcal^b(coh(X))$.
\end{remark}
\end{section}
\begin{section}{Perverse quasi-coherent t-structures for noncommutative projective planes}

The aim of this section is to use the construction of section 3 to create an analogue of perverse coherent t-structures in the derived categories of certain noncommutative projective planes, motivated by the comparison established on section 4 in the commutative case. This entails finding an injective cogenerator in $Tails(R)$ for a suitable class of $\mathbb{K}$-algebras $R$ and set up a definition of perversity that generalises the commutative one.

\begin{remark}\label{conj restriction}
$tails(R)$ is not cocomplete and therefore, in this section, we can only do the construction of section 3 in $Tails(R)$ (hence the word \textit{quasi-coherent} rather than \textit{coherent} in the title). However, taking into account theorem \ref{main perverse}, we conjecture that indeed the constructions in this section restrict well to $\Dcal^b(tails(R))$.
\end{remark}

We shall focus on the case where $R$ is a graded elliptic 3-dimensional Artin-Schelter regular algebra which is finite over its centre. These algebras are interesting for our purposes since they are fully bounded noetherian (more than that, they are PI, as proved in \cite{ATV2}). Also, a graded noetherian algebra which is fully bounded is graded fully bounded (\cite{VOV}). This is important for the following result that allows us to parametrise a useful collection of injective objects via prime ideals. In this sense, although these examples are noncommutative, we are still very close to the commutative setting (\cite{Matlis}).

Recall that there is a map from the set of indecomposable injective graded modules to the set of homogeneous prime ideals given by assigning to an injective $E$ its homogeneous assassinator ideal, $Ass(E)$. The assassinator ideal of an indecomposable object is the only prime ideal associated to $E$, i.e., the only prime ideal which is maximal among the annihilators of nonzero submodules of $E$ (and there is a natural graded version of this concept - see \cite{NVO} and \cite{VOV}).

\begin{proposition}[Natasescu, Van Oystaeyen, \cite{NVO}, Theorem C.I.3.2]
Let $R$ be a positively graded noetherian ring. Then $R$ is graded fully bounded if and only if the map that assigns the corresponding assassinator ideal to an indecomposable injective graded module induces a bijection between indecomposable injective modules in $Gr(R)$ (up to isomorphism and graded shift) and homogeneous prime ideals of $R$.
\end{proposition}

\begin{remark}\label{bij}
In the context of this proposition, the indecomposable injective associated with a homogeneous prime $P$ is the unique (up to isomorphism and shifts) indecomposable direct summand of $E^g(R/P)$ (\cite{NVO}, Theorem C.I.3.2), thus establishing an inverse map.
\end{remark}

This result brings us closer to the desired cogenerating set. Its significance in our context comes from the work of Matlis on the decomposition of injective modules over noetherian rings. Matlis proved that $R$ is (right) noetherian if and only if every injective (right) module is the direct sum of indecomposable injective (right) modules (\cite{Matlis}). This shows in particular that the set of indecomposable injective objects cogenerates the category of modules over a noetherian ring. There is a graded analogue of this result, as follows.

\begin{proposition}[Samir Mahmoud, \cite{SM}]\label{dec}
Let $M$ be a finitely generated graded module over a graded noetherian ring $R$. Then $E^g(M)$ is a finite direct sum of indecomposable injective objects in $Gr(R)$.
\end{proposition}

These two results yield a useful set of torsion classes parametrised by the homogeneous prime ideals of $R$. Indeed, for an indecomposable injective $E$, denote the corresponding torsion class cogenerated by $E$ in $Gr(R)$ by $\mathcal{T}_E$, i.e.,
\begin{equation}\nonumber
\mathcal{T}_E:=\left\{M\in Gr(R):\ Hom_{Gr(R)}(M, E)=0\right\}.
\end{equation} 
Let $\widehat{Y}$ denote the set of all such classes where $E$ runs over indecomposable injective objects, up to isomorphism and graded shift, such that its assassinator ideal is not the irrelevant ideal, i.e., $Ass(E)\neq R_+$. 
Analogously define $Y$ to be the set of the torsion classes in $Tails(R)$ of the form 
$$\mathcal{T}_{\pi E}:=\left\{K\in Tails(R):\ Hom_{Tails(R)}(K, \pi E)=0\right\}$$  for all $E$ indecomposable injective graded module with non-irrelevant assassinator.

\begin{corollary}\label{int torsion cl}
Let $R$ be a positively graded fully bounded connected noetherian $\mathbb{K}$-algebra generated in degree 1. Then the intersection (in $Tails(R)$) of the torsion classes in $Y$ is zero.
\end{corollary}
\begin{proof}
Suppose $\pi M$ lies in the intersection of the torsion classes of $Y$. Then by lemma \ref{hom tails 0}, $M$ lies in the intersection of the torsion classes in $\hat{Y}$. By propostion \ref{dec}, the indecomposable injective objects cogenerate $Gr(R)$ and thus $E^g(M)$ must be a finite direct sum of the indecomposable injective associated with $R_+$. This indecomposable is a direct summand of $E^g(R/R_+)$ (see remark \ref{bij}), whose projection in $Tails(R)$ is therefore zero. Thus $\pi E^g(M)=0$ and so $\pi M=0$.

\end{proof}

We proceed now to the desired construction. Let $R$ be a  positively graded Artin-Schelter regular algebra of dimension 3 generated in degree one which is finitely generated over its centre. As discussed before, it is graded fully bounded noetherian. We need to define a perversity in $Y$, where $Y$ is as before.

\begin{definition}
A perversity is a map $p: Y\longrightarrow \mathbb{Z}$ such that, given $\mathcal{T}_{\pi E_1}$, $\mathcal{T}_{\pi E_2}$ in $Y$, if there is a nonzero homomorphism from $\pi E_2$ to $\pi E_1$ then
\begin{equation}\nonumber
p(\mathcal{T}_{\pi E_1})-(GKdim(R/Ass(E_2))-GKdim(R/Ass(E_1))) \leq p(\mathcal{T}_{\pi E_2}) \leq p(\mathcal{T}_{\pi E_1}).
\end{equation}
\end{definition}

We now prove that this definition of perversity coincides, when the algebra is commutative, with the definition of perversity of the introduction. We start by a supporting lemma.

\begin{lemma}\label{pts vs injs}
Let $R$ be a positively graded commutative noetherian $\mathbb{K}$-algebra and $X=Proj(R)$. The following are equivalent.
\begin{enumerate}
\item For $x_1,x_2\in X^{top}$, $x_1\in \bar{x}_2$;
\item $P_2:=Ann(x_2)\subset Ann(x_1)=:P_1$, where $Ann(x_i)$ denotes the homogeneous ideal of functions vanishing in $x_i$;
\item There is a nonzero homomorphism from $R/P_2$ to $R/P_1$;
\item There is a nonzero homomorphism from $E^g(R/P_2)$ to $E^g(R/P_1)$.
\end{enumerate}
\end{lemma}
\begin{proof}
It is clear that $(1)\Leftrightarrow (2) \Rightarrow (3) \Rightarrow (4)$. 
We only need to prove $(4)\Rightarrow (2)$. Let $f$ be a homomorphism from $E^g(R/P_2)$ to $E^g(R/P_1)$ and $a\in h(P_2)\setminus h(P_1)$. Clearly $N:=R/P_1\cap im(f)\neq 0$ since $R/P_1$ is a graded essential submodule of $E^g(R/P_1)$. Now, $N\cap (a+P_1)R\neq 0$ since any homogeneous ideal of a commutative graded domain is graded essential (the product of two nonzero ideals is nonzero and it is contained in the intersection). Hence, $0\neq(a+P_1)R\cap N \subset (a+P_1)R\cap im(f)$. Let then $b$ be a nonzero element in $(a+P_1)R\cap im(f)$ and $y\in E^g(R/P_2)$ such that $b=ar+P_1 = f(y)$. Note that $ya\in P_2E^g(R/P_2)$ and $P_2$ annihilates $E^g(R/P_2)$, thus $0=f(ya)=a^2r+P_1$ and $r\in P_1$, since $P_1$ is prime. Hence $b=0$ in $R/P_1$, reaching a contradiction and proving the result.
\end{proof}

\begin{proposition}
If $R$ is a positively graded noetherian connected commutative $\mathbb{K}$-algebra generated in degree 1, the definition of perversity above is equivalent to the commutative definition of perversity in equation (\ref{def perv}).
\end{proposition}
\begin{proof}
Let $X$ be the projective scheme associated with $R$. Note that points $x\in X^{top}$ are in bijection with homogeneous prime ideals not equal to the irrelevant ideal of $R$ and these are in bijection with graded torsion-free indecomposable injectives in $Gr(R)$ (up to isomorphisms and shifts). Suppose $x_1, x_2\in X^{top}$, $P_1, P_2$ the associated homogeneous prime ideals and $E_1, E_2$ the corresponding indecomposable injectives. The condition $x_1\in \bar{x}_2$ translates into the existence of a nonzero map from $E_2$ to $E_1$ by lemma \ref{pts vs injs} (note that, in this case, $E_i=E^g(R/P_i)$ since $R/P$ is indecomposable in $Gr(R)$ and hence so is its injective envelope) and by lemma \ref{hom tails 0} this is equivalent to the existence of a map from $\pi E_2$ to $\pi E_1$.

Since $R$ is finitely generated over $\mathbb{K}$ (as it is noetherian), and hence are all its quotients, it is known the Krull dimension of $R/P_i$ (which is the same as $dim(x_i)$ in the geometric definition of perversity - see introduction) coincides with the Gelfand-Kirillov dimension of $R/P_i$ (\cite{KL}, Theorem 4.5).  
The result then follows by making the adequate substitutions in equation (\ref{def perv}).
\end{proof}

Recall that 3-dimensional Artin-Schelter regular algebras are noetherian domains - in particular, they are prime rings (\cite{ATV}, \cite{ATV2}). This allows us to prove the following useful lemma.

\begin{lemma}
Let $R$ be a positively graded connected 3-dimensional Artin-Schelter regular algebra generated in degree 1 which is finitely generated over its centre. Then the image of a perversity $p$ as defined above is finite.  
\end{lemma} 
\begin{proof}
Since $R$ is prime, $(0)$ is a prime ideal not equal to the irrelevant ideal. Thus it corresponds to an indecomposable injective object which we denote by $E_0$. Furthermore, as a consequence of remark \ref{bij}, $E^g(R)$ is a finite direct sum of copies of $E_0$. Similarly, $E^g(R/P)$ is a finite direct sum of copies of $E_P$, the indecomposable injective object associated to the homogeneous prime ideal $P$. We observe that for any such $P$, there is a map from $E^g(R)$ to $E^g(R/P)$ induced by the canonical projection from $R$ to $R/P$. Therefore, there is a nontrivial map from $E_0$ to $E_P$. The perversity condition then assures that:  
\begin{equation}\nonumber
p(\mathcal{T}_{\pi E_P})-(GKdim(R)-GKdim(R/P)) \leq p(\mathcal{T}_{\pi E_0}) \leq p(\mathcal{T}_{\pi E_P}).
\end{equation}
Since, by definition, the Gelfand-Kirillov dimension of $R$ is finite (and so is the dimension of any of its quotients - \cite{KL}, Lemma 3.1) we have that, for a fixed value of $p(\mathcal{T}_{\pi E_0})$, 
$p(\mathcal{T}_{\pi E_P})$ is an integer that differs at most $GKdim(R)$ from it. Hence the image of $p$ is finite.
\end{proof}
Thus, for $R$ Artin-Schelter regular algebra of dimension 3 and finite over its centre, we can form a finite chain of hereditary torsion classes of $Tails(R)$, compactly generated in $\Dcal^b(Tails(R))$:
\begin{equation}\nonumber
S:=\left\{\Tcal_i:=\bigcap\limits_{T: p(T)\leq i}T, min(p)\leq i\leq max(p)\right\}.
\end{equation}
By corollary \ref{int torsion cl}, the last element of the chain, $\Tcal_{max(p)}$, is zero. Finally, section 3 provides a way of building a perverse quasi-coherent t-structure with respect to $p$ by defining its aisle to be $\Dcal^{S,\leq 0}$ in $\Dcal^b(Tails(R))$. As mentioned in remark \ref{conj restriction}, in light of section 4, we conjecture that these t-structures restrict to $\Dcal^b(tails(R))$.
\end{section}

\end{document}